\documentclass[11pt,reqno]{amsart}
\usepackage[T1]{fontenc}
\usepackage{graphicx}

\usepackage{color}
\definecolor{MyLinkColor}{rgb}{0,0,0.4}

\newcommand{\R}{{\mathbb R}}

\newcommand{\N}{{\mathbb N}}

\newcommand{\s}{\mathbb{S}}

\newcommand{\cF}{\mathcal{F}}

\newcommand{\kL}{\mathcal{L}}
\newcommand{\cC}{\mathcal{C}}

\newcommand{\la}{\langle}
\newcommand{\ra}{\rangle}

\newcommand{\ov}{\overline}

\newcommand{\p}{\partial}

\newcommand{\e}{\varepsilon}

\newcommand{\0}{\Omega}

\newcommand{\tr}{\mathop{\rm tr}\nolimits}

\newcommand{\spa}{\mathop{\rm span}\nolimits}
\newcommand{\im}{\mathop{\rm Im}\nolimits}
\newcommand{\ke}{\mathop{\rm Ker}\nolimits}

\newtheorem{thm}{Theorem}[section]
\newtheorem{prop}[thm]{Proposition}
\newtheorem{defn}[thm]{Definition}
\newtheorem{lemma}[thm]{Lemma}
\newtheorem{cor}[thm]{Corollary}
\theoremstyle{remark} 
\newtheorem{rem}[thm]{Remark}

\numberwithin{equation}{section}

\title[Capillary-gravity water waves with bounded vorticity]{Capillary-gravity water waves with discontinuous vorticity: existence and regularity results}

\author[A.--V. Matioc]{Anca--Voichita Matioc}
\address{Institut f\" ur Mathematik, Universit\" at Wien, Nordbergstra{\ss}e 15,
1090 Wien, Austria}
\email{anca.matioc@univie.ac.at}

\author[B.--V. Matioc]{Bogdan--Vasile Matioc}
\address{Institut f\" ur Mathematik, Universit\" at Wien, Nordbergstra{\ss}e 15,
1090 Wien, Austria}
\email{bogdan-vasile.matioc@univie.ac.at}


\usepackage[colorlinks=true,linkcolor=MyLinkColor,citecolor=MyLinkColor]{
hyperref} 

\begin{document}

\begin{abstract}
In this paper we construct  periodic capillarity-gravity water waves with an arbitrary  bounded  vorticity distribution.
This is achieved by reexpressing, in the height function formulation  of the water wave problem, the boundary condition obtained from Bernoulli's principle   as a  
nonlocal differential equation. 
This enables us to establish the existence of weak solutions of the problem by using  elliptic estimates and bifurcation theory. 
Secondly, we investigate the a priori regularity  of these weak solutions and  prove that they are in fact strong solutions of the problem, describing waves with a 
real-analytic free surface.
Moreover, assuming merely integrability of the vorticity function, we show that any weak solution  corresponds to flows having real-analytic streamlines.
\end{abstract}

\maketitle

\section{Introduction}\label{Sec:1}

This paper is concerned  with periodic  capillary-gravity water waves traveling over a homogeneous fluid and having an arbitrary bounded vorticity distribution.
Our study is motivated by the physical setting of wind generated waves  which  possess a thin layer of high vorticity   \cite{PB74}, or even high vorticity regions 
beneath the wave crests  \cite{O82}.
On the other hand, in the near-bed region there may exist strong tidal currents which  interact with the water waves and contribute so to  the transportation of sediments \cite{PC84}. 
The plethora of  phenomena resulting from the  wave-current interactions makes the study of rotational water waves so interesting,  cf. \cite{Con11, Jon, Thom}.
Indeed, for irrotational waves in the absence of an underlying
current the fluid velocity, the pressure, and the particle
paths in the flow present very regular features that can
be described qualitatively even for waves of large
amplitude (see \cite{C12, Co06, Um12}). However, already within the
setting of irrotational steady waves with an
underlying uniform current one encounters new particle
path patterns, cf. \cite{CoSt10, HH12}, while the behavior of the
velocity field and of the pressure is considerably altered
by an underlying current of constant non-zero vorticity.
For rotational waves the most dramatic changes (in the form of critical layers)
are triggered by the presence of stagnation points in the
flow but  even in the absence of stagnation points
significant changes occur (see \cite{ SP88}). A discontinuous
vorticity enhances these departures from features that
hold within the irrotational regime, as indicated by
the numerical simulations in  \cite{KO1,KO2}.

On the basis of a rigorous theory,  exact periodic gravity water waves with a discontinuous vorticity have been shown to exist in \cite{CS11} by making use of a weak 
formulation of the water wave problem.
Subsequently, capillary-gravity water waves interacting with several vertically superposed and linearly sheared currents of different  vorticities have been constructed in \cite{CM13xx}, 
by regarding the height function formulation of the hydrodynamical problem as 
a diffraction problem.
We develop herein a rigorous existence theory for capillary-gravity water waves with a bounded general  vorticity function, some of the analysis   in \cite{CM13xx}
serving as a preliminary step.   

The existence of exact  capillary-gravity water waves was first established in the irrotational setting  \cite{MJ89, JT85, JT86, RS81}, the existence theory  
for rotational waves being developed more recently
in the setting of waves with constant vorticity, stagnation points, and possibly with overhanging profiles  \cite{CM12x}  (see also \cite{CV11, CM13xxx}),
or for waves with a general H\"older continuous vorticity distribution  \cite{W06b}.
Many papers are also dedicated to the study of the properties of  capillary-gravity water waves and of the flow beneath  them, such as the regularity of the
wave profile and that of the streamlines \cite{Hen10, DH11a, LW12x, AM12x, WZ12},
or the description of the particle paths \cite{DH07}. 

The first goal of this paper is to establish the existence of two-dimensional capillary-gravity water waves with an  arbitrary bounded  vorticity and without stagnation points.
This is achieved by using the height function formulation of the water wave problem and by defining a suitable notion of weak solution for this problem.
Reexpressing the boundary condition obtained from Bernoulli's law as a nonlocal boundary condition, we   obtain a  new  equivalent formulation of the problem 
which enables us to consider the  existence problem of weak solutions
in an abstract bifurcation setting.
Using elliptic theory \cite{GT01} and local bifurcation tools \cite{CR71}, we then establish the existence of infinitely many bifurcation branches consisting 
of non-laminar weak solutions of the hydrodynamical problem.
Our second goal is to determine the a priori regularity properties  of the weak  solutions in the case when the vorticity function is merely integrable.
This problem is in the setting of rotational waves very recent \cite{AC11, Esch-reg12}, but its implications are very important when studying the symmetry properties of water waves.
More precisely, in view of  \cite[Theorem  3.1  and Remark 3.2]{MM13} and \cite[Corollary 1.2]{EM13x} the following statement holds true:
\begin{itemize}\item[] \em Within the set of all periodic gravity waves without stagnation points
 the symmetric waves with one crest and trough per period 
are  characterized by the property that all the streamlines have a global minimum on the same vertical line.
\end{itemize}
We emphasize that   the   gravity waves with only one crest and trough per period are symmetric waves \cite{CoEhWa07, CoEs04_b, okamoto-shoji-01}.
The availability of Schauder estimates for the new formulation of the problem stands at the basis of our regularity result where we state that  the streamlines and 
the wave profiles corresponding to such  weak solutions are real-analytic graphs.
As a particular case, we establish the real-analyticity of the streamlines also for  pure capillary water waves, generalizing previous results \cite{DH12, LW12x}.
Our regularity result could serve as a tool when studying the symmetry of waves with capillary effects, the symmetry   problem being in this setting still open.
Besides, the additional regularity properties of the weak solutions help us to prove that the weak solutions that we have found are in fact strong solutions, and 
even classical if the vorticity function is   continuous.

The outline of the paper is as follows: we start the   Section \ref{Sec:2} by presenting the   mathematical model and  state at the end the main existence result Theorem \ref{MT}.
In Section \ref{Sec:3} we derive a new formulation of the water wave problem, which we recast in Section \ref{Sec:4} as a nonlinear and nonlocal problem. 
After proving the existence of local bifurcation branches of weak solutions for the  latter problem in Theorem \ref{T:LB}, we study in Section \ref{Sec:5} 
the a priori regularity of the weak solutions  when assuming merely integrability of the vorticity function, cf. 
Theorem \ref{MT2} and Corollary \ref{C:1}.
We conclude the paper with the proof of Theorem \ref{MT}.

\section{The mathematical model and the existence result}\label{Sec:2}

\paragraph{\bf The mathematical model}
We start by presenting three equivalent mathematical models which describe the propagation of  periodic   water waves   over a rotational, 
inviscid, and incompressible fluid, under the influence of  gravity and capillary  forces.
The waves that we consider are  two-dimensional and they travel  at constant  speed $c>0$.
In a reference frame which moves in the same direction as the wave and with the same speed $c$
 the equations of motion are the steady-state Euler equations
\begin{subequations}\label{eq:P}
  \begin{equation}\label{eq:Euler}
\left\{
\begin{array}{rllll}
({u}-c) { u}_x+{ v}{ u}_y&=&-{ P}_x,\\
({ u}-c) { v}_x+{ v}{ v}_y&=&-{ P}_y-g,\\
{ u}_x+{v}_y&=&0.
\end{array}
\right.\qquad \text{in $\0_\eta.$}
\end{equation}
We have assumed that the free surface of the wave is the graph $y=\eta(x),$ that the fluid has constant  density, set to be   $1$,  and that the fluid bed is located at $y=-d$. 
Hereby, $d>0$ is the average mean depth of the fluid, meaning that  the fluid domain is 
\[
\0_\eta:=\{(x,y)\,:\,\text{$ x\in\s $ and $-d<y<\eta(x)$}\},
\]
whereby $\s$ is the unit circle.
This notation is used to express the fact that the  function $\eta,$ the velocity field $ (u, v), $ and the pressure $P$ are $2\pi$-periodic in $x.$
Since we incorporate the effect of surface tension in our problem, the equations  \eqref{eq:Euler} are supplemented by the following boundary conditions
\begin{equation}\label{eq:BC}
\left\{
\begin{array}{rllll}
{ P}&=&{P}_0-\sigma\eta''/(1+\eta'^2)^{3/2}&\text{on $ y=\eta(x)$},\\
{ v}&=&({ u}-c) \eta'&\text{on $ y=\eta(x)$},\\
{ v}&=&0 &\text{on $ y=-d$},
\end{array}
\right.
\end{equation}
with ${ P}_0$ denoting the constant atmospheric pressure and $\sigma>0$ being the surface tension coefficient. 
Moreover, the  vorticity of the flow is the scalar function 
\begin{equation}\label{vor}
\omega:= { u}_y-{ v}_x\qquad\text{in $\ov\0_\eta$.}  
\end{equation}
\end{subequations}
The goal of this paper is to prove  the existence   of solutions of the problem \eqref{eq:P} in the class
\begin{align}\label{Reg1}
\eta\in C^{2-}(\s), \quad u,v,P\in C^{1-}(\ov\0_\eta), \quad \omega\in L_\infty(\0_\eta),
\end{align}
and to study their additional regularity properties. 
Hereby we may identify the spaces  $C^{k-}(\s) $ and $C^{k-}(\ov\0_\eta)$,   $1\leq k\in \N$, which contain functions that have Lipschitz continuous  derivatives of order $k-1$, 
with $W^k_\infty(\s) $ and $W^{k}_\infty(\0_\eta)$, respectively, cf. \cite{EG92}.

The problem \eqref{eq:P} can also be formulated  in terms of  the stream function $\psi:\ov\0_\eta\to\R$, which is given by 
\[
\psi(x,y):=-p_0+\int_{-d}^y({ u}(x,s)-c)\, ds\qquad \text{for $(x,y)\in\ov\0_\eta$}.
\]
It follows readily from this formula  that $\psi\in C^{2-}(\ov\0_\eta)$ satisfies $\nabla\psi=(-{ v},{ u}-c).$ 
Additionally, it can be shown that the problem \eqref{eq:P} is equivalent to the following  free boundary problem
\begin{equation}\label{eq:psi}
\left\{
\begin{array}{rllll}
\Delta \psi&=&\gamma(-\psi)&\text{in}&\0_\eta,\\
|\nabla\psi|^2+2g(y+d)-\displaystyle2\sigma\frac{\eta''}{(1+\eta'^2)^{3/2}}&=&Q&\text{on} &y=\eta(x),\\
\psi&=&0&\text{on}&y=\eta(x),\\
\psi&=&-p_0&\text{on} &y=-d,
\end{array}
\right.
\end{equation}
cf. \cite{Con11, CoSt04, BM11}.
We emphasize that the first  boundary condition in \eqref{eq:psi} is obtained from Bernoulli's  principle which states that the total energy
\[E:=\frac{(u-c)^2+v^2}{2}+g(y+d)+P-\int_0^{\psi}\gamma(-s)\, ds\]
is constant in $\ov\0_\eta.$
In \eqref{eq:psi}, the constant $p_0<0$ represents the relative mass flux,   $Q\in\R$ is related to the so-called total head,
and the function $\gamma$ is the vorticity function.
The existence of the vorticity function is obtained under the additional assumption that the horizontal velocity of each fluid particle is less than the wave speed  
\begin{equation}\label{eq:cond1}
{ u}-c<0\qquad\text{in $\ov \0_\eta$.}
\end{equation}
Indeed, the relation \eqref{eq:cond1} together with \eqref{eq:Euler} imply, cf. \cite{CoSt04, BM11}, that there exists a function $\gamma\in L_\infty((p_0,0))$ such that 
$\omega(x,y)=\gamma(-\psi(x,y))$ almost everywhere in $\0_\eta$.
 
Assuming \eqref{eq:cond1},  the  stream function formulation \eqref{eq:psi} can be reexpressed in terms of the so-called height function.
Indeed, the assumption \eqref{eq:cond1}, ensures that the mapping
 $\Phi:\ov\0_\eta\to\ov\0$  given by
\[
\Phi(x,y):=(q,p)(x,y):=(x,-\psi(x,y))\qquad \text{for $(x,y)\in\ov\0_\eta$},
\]
whereby $\0:=\s\times(p_0,0),$ 
is a diffeomorphism of class $C^{2-}$.
Consequently, the height function   $h:\ov \0\to\R$ defined by  $h(q,p):=y+d$ for $(q,p)\in\ov\0$ belongs to $C^{2-}(\ov\0)=W^2_\infty(\0)$ and
 it solves the nonlinear boundary value problem
\begin{equation}\label{PB}
\left\{
\begin{array}{rllll}
(1+h_q^2)h_{pp}-2h_ph_qh_{pq}+h_p^2h_{qq}-\gamma(p)h_p^3&=&0&\text{in $\0$},\\
\displaystyle 1+h_q^2+(2gh-Q)h_p^2-2\sigma \frac{h_p^2h_{qq}}{(1+h_q^2)^{3/2}}&=&0&\text{on $p=0$},\\
h&=&0&\text{on $ p=p_0,$}
\end{array}
\right.
\end{equation}
together with  the condition 
\begin{equation}\label{PBC}
 \min_{\ov \0}h_p>0.
\end{equation}
We stress at this point   that the function $h$ associates to each point $(q,p)\in\ov\0$ the value of the height of fluid particle $(x,y)=\Phi^{-1}(q,p)$ above the flat bed.
Particularly, the wave profile is parametrized by the map $\eta=h(\cdot,0)-d,$ implying   that $h(\cdot,0)\in C^{2-}(\s)$.
With this observation,  the boundary condition of \eqref{PB} on $p=0$ is also meaningful.
In fact, each streamline of the steady flow corresponds to a level curve of $\psi$ and is therefore  parametrized by the function  $h(\cdot,p)-d$, whereby $p\in[p_0,0]$
is uniquely determined by the streamline.
Moreover, as a direct consequence of \eqref{PBC}, that the first  equation of \eqref{PB} is  uniformly elliptic.
The equivalence of the problems \eqref{eq:P}, \eqref{eq:psi}, and \eqref{PB} under the assumption \eqref{eq:cond1} (or equivalently \eqref{PBC}) in the $W^2_\infty-$setting
follows easily from previous contributions \cite{Con11, Esch-reg12, BM11} (see also \cite{CS11}).

Our main existence result is the following theorem.
\begin{thm}[Existence result]\label{MT}
Let  $\gamma\in L_\infty((p_0,0))$ be given.
Then,   there exists a positive integer $n$ and    connected curves $\cC_k,$ $k\in\N\setminus\{0\},$ consisting only of  solutions of the problem
 \eqref{PB}-\eqref{PBC} with the property that each solution $h$ belonging to one of the curves satisfies
 \begin{itemize}
 \item[$(i)$] $h\in W^2_\infty(\0)$,
 \item[$(ii)$] $h(\cdot,p)$ is a real-analytic map for all $p\in[p_0,0].$ 
 \end{itemize}
 Each curve $\cC_k$  contains a laminar flow (all the streamlines being parallel to the flat bed)
 and all the other points on the curve correspond to solutions that have minimal period $2\pi/(kn) $, only one crest and trough per period, and are symmetric with respect to the crest line.  
\end{thm}

\begin{rem}\label{R:0}
The integer $n$ in Theorem \ref{MT} may be chosen to be $n=1$ provided that the condition \eqref{d2} is satisfied.
\end{rem}

We emphasize that the regularity property $(ii)$ of the solutions found  in Theorem \ref{MT} guarantees that the wave surface and all the streamlines of the flows are real-analytic graphs.
Additional regularity properties of the solutions found in Theorem \ref{MT} are derived in Section \ref{Sec:5}, cf. Theorem \ref{MT2}.

\section{A fourth equivalent  formulation for  the   water wave   problem}\label{Sec:3}
The main difficulty in proving the existe\begin{equation}
 \Gamma(p):=\int_0^p\gamma(s)\, ds\qquad\text{for $p_0\leq p\leq0$,}
\end{equation}nce of solutions of problem \eqref{PB} is due to the fact that we have to deal with a quasilinear elliptic equation  in a $W^2_\infty-$setting.
 A further impediment is raised by the nonlinear boundary condition on $p=0$ which contains second order derivatives of the unknown.
Therefore, we cannot attack  \eqref{PB} directly.
Instead, we consider a weak formulation of \eqref{PB} and establish first the existence of weak solutions to \eqref{PB} that satisfy \eqref{PBC}.
Later on, we improve the regularity of these weak  solutions and show finally that they are  the strong solutions found in Theorem \ref{MT}. 
To this end, we introduce the anti-derivative   $\Gamma:[p_0,0]\to\R$ of $\gamma$ by the relation
\begin{equation}
 \Gamma(p):=\int_0^p\gamma(s)\, ds\qquad\text{for $p_0\leq p\leq0$,}
\end{equation}
and observe that the first equation of \eqref{PB} can be recast in the equivalent form
\[\left(\frac{h_q}{h_p}\right)_q-\left(\Gamma+\frac{1+h_q^2}{2h_p^2}\right)_p=0\qquad\text{in $\0$.}\]
This relation motivates  us  to introduce the following notion of weak solution of \eqref{PB}.
\begin{defn}\label{D:1} A function $h\in C^{1 }(\ov\0)$ is called a {\em weak solution} of \eqref{PB}  if
\begin{itemize}
 \item[$(i)$] $h(\cdot,0)\in C^{2 }(\s)$; 
\item[$(ii)$] $h$ satisfies both boundary conditions of \eqref{PB};
\item[$(iii)$] $h$ satisfies the following integral equation
\begin{equation}\label{PB1}
 \int_\0\frac{h_q}{h_p}\phi_q-\left(\Gamma+\frac{1+h_q^2}{2h_p^2}\right)\phi_p\,  d(q,p)=0\qquad\text{for all $\phi\in C^1_0(\0)$.}
\end{equation}
\end{itemize}
\end{defn}
We have denoted by $ C^1_0(\0)$ the space containing continuously differentiable  functions   with compact support in $\0.$
It is easy to see that any classical solution of the problem \eqref{PB}-\eqref{PBC}, cf. \cite{W06b}, is also a weak solution of this problem.
The disadvantage of this definition is that one needs to require more regularity from $h$ on the boundary $p=0,$ fact which makes it difficult to consider a
 suitable functional analytic setting for this concept of weak solutions. 
Fortunately, we can recast the second order  boundary condition of \eqref{PB}, which is obtained from to Bernoulli's principle,   as a nonlinear and nonlocal equation.
This operation  has the benefit of transforming the boundary condition from a  differential equation of order two--we lose two derivatives due to the 
curvature term--into a nonlocal equation of order zero.

In the following $\tr_0$ will denote the trace operator with respect to boundary  $p=0,$ that is $\tr_0v=v(\cdot,0)$ for all $v\in C(\ov\0).$
Let $\alpha\in(0,1)$ be fixed for the remainder of the paper.
\begin{lemma}\label{L:1} Let $(1-\p_q^2)^{-1}\in \kL(C^{\alpha}(\s), C^{2+\alpha}(\s))$ denote the inverse of the linear operator 
\[1-\p_q^2:C^{2+\alpha}(\s)\to C^{\alpha}(\s).\]
\begin{itemize}
 \item[$(i)$] Assume that $h \in C^{1+\alpha }(\ov\0)$  is a weak solution of    \eqref{PB} that satisfies  additionally the condition  \eqref{PBC}.
Then, $h$  also satisfies the following equation
\begin{equation}\label{NBC}
h+(1-\p_q^2)^{-1}\tr_0\left( \frac{\left(1+h_q^2+(2gh-Q)h_p^2\right)(1+h_q^2)^{3/2}}{2\sigma h_p^2}-h\right)=0 \qquad\text{on $p=0$.}
\end{equation}
\item[$(ii)$] Assume that $h\in C^{1+\alpha}(\ov\0)$   verifies   the condition  \eqref{PBC} and that $h$ is a weak solution of the problem
\begin{equation}\label{PB'}
\left\{
\begin{array}{rllll}
(1+h_q^2)h_{pp}-2h_ph_qh_{pq}+h_p^2h_{qq}-\gamma(p)h_p^3&=&0&\text{in $\0$},\\[1ex]
\displaystyle h+(1-\p_q^2)^{-1}\tr_0\left( \frac{\left(1+h_q^2+(2gh-Q)h_p^2\right)(1+h_q^2)^{3/2}}{2\sigma h_p^2}-h\right)&=&0&\text{on $p=0$},\\[1ex]
h&=&0&\text{on $ p=p_0,$}
\end{array}
\right.
\end{equation}
that is $h$ satisfies the last two equations of \eqref{PB'} pointwise and the first equation in the weak sense defined in Definition \ref{D:1} $(iii)$.
Then, $h$ is a weak solution of \eqref{PB}. 
\end{itemize}
\end{lemma}
\begin{proof}
It is easy to see that $(ii)$ follows from $(i)$.
On the other hand, if we assume that $(ii)$ is satisfied, we only need to show that $\tr_0 h\in C^{2+\alpha}(\s)$ and that $h$ satisfies the second boundary condition of \eqref{PB}.
Noticing that the second equation of \eqref{PB'} implies
\[
\tr _0h =-(1-\p_q^2)^{-1}\tr_0\left(\frac{\left(1+h_q^2+(2gh-Q)h_p^2\right)(1+h_q^2)^{3/2}}{2\sigma h_p^2}-h\right)
\] 
 we deduce that $\tr_0h\in C^{2+\alpha}(\s),$ and,  applying the operator $(1-\p_q^2)$ to the latter equation, we obtain the desired conclusion. 
\end{proof}

The advantage of the formulation \eqref{PB'} of the water wave problem is that all its equations are well-defined for functions  $h\in C^{1+\alpha}(\ov\0)$.
This allows  us to introduce a functional analytic setting and to recast \eqref{PB'} as a bifurcation problem.
Then, using the theorem on local bifurcation from simple eigenvalues due to Crandall and Rabinowitz \cite{CR71} we determine weak solutions of \eqref{PB'} and \eqref{PBC} which are located on real-analytic curves.

\section{Local bifurcation of weak solutions}\label{Sec:4}

We now introduce  a parameter $\lambda$ into the problem \eqref{PB'} which is   used to describe  the trivial solutions of \eqref{PB'}.
These are laminar flows, with a flat surface and parallel streamlines, and are denoted by $H.$
Indeed, if $H\in C^{1+\alpha}(\ov\0)$ is a weak solution of \eqref{PB'} and \eqref{PBC} which is independent of  $q$,  then $H(p_0)=0$,  
\begin{equation*}
H(0)+(1-\p_q^2)^{-1}\left( \frac{ 1+(2gH(0)-Q)H_p^2(0) }{2\sigma H_p^2(0)}-H(0)\right)=0, 
\end{equation*}
and
\begin{equation*}
 \int_\0\left(\Gamma+\frac{1}{2H_p^2}\right)\phi_p\,  d(q,p)=0\qquad\text{for all $\phi\in C^1_0(\0)$.}
\end{equation*}
This last relation ensures that $\Gamma+1/(2H_p^2)$ is a constant function. 
Taking into account that $(1-\p_q^2)^{-1}c=c$ for all $c\in\R,$
we find a constant $\lambda>2\max_{[p_0,0]}\Gamma $ such that we have
\begin{equation}\label{E:LFS}
\begin{aligned}
H(p)&:=H(p;\lambda):=\int_{p_0}^p\frac{1}{\sqrt{\lambda-2\Gamma(s)}}\, ds,\qquad\text{ $p\in[p_0,0]$}.
\end{aligned}
\end{equation}
Requiring that   $H$   solves also the boundary condition on $p=0,$  we determine the head $Q$ as a function of the parameter $\lambda$ 
\begin{equation}\label{Q}
Q:=Q(\lambda):=\lambda+2g\int_{p_0}^0\frac{1}{\sqrt{\lambda-2\Gamma(p)}}\, dp.
\end{equation}
Let us observe that in fact $H\in C^{2-}([p_0,0])$, and that  the constant $\lambda$ is related to the horizontal speed at the top of the laminar flow by the relation
\[\sqrt{\lambda}=\frac{1}{H_p(0)}=(c-{ u})\big|_{y=0}.\]

We now present an abstract  functional analytic setting which allows us to recast  the problem \eqref{PB'} as an operator equation.
We choose  therefore an integer  $n\in\N$ with $n\geq1$ (which will be fixed later on) and define the Banach spaces:
\begin{align*}
 X&:=\left\{h\in C^{1+\alpha}_{2\pi/n}(\ov\0)\,:\, \text{$h$ is even in $q$ and $h\big|_{p=p_0}=0$}\right\},\\
 Y_1&:=\{f\in\mathcal{D}'(\0)\,:\, \text{$f=\p_q\phi_1+\p_p\phi_2$ for  $\phi_1,\phi_2\in C^\alpha_{2\pi/n}(\ov\0)$ with $\phi_1$ odd and $\phi_2$ even in $q$}\},\\
 Y_2&:=\{\varphi\in C^{1+\alpha}_{2\pi/n}(\s)\,:\, \text{$\varphi$ is even}\},
\end{align*}
whereby we have identified, when defining $Y_2$, the unit circle $\s$  with the line $p=0.$ 
The subscript ${2\pi/n}$ means that we consider functions which are $2\pi/n$--periodic only.
The space  $Y_1$ is a Banach space with the norm
\[
\|f\|_{Y_1}:=\inf\{\|\phi_1\|_\alpha+\|\phi_2\|_\alpha\,:\, f=\p_q\phi_1+\p_p\phi_2\}.
\]
Moreover, we introduce the operator $\cF:=(\cF_1,\cF_2):(2\max_{[p_0,0]}\Gamma,\infty)\times X\to Y:=Y_1\times Y_2$ by the relations
\begin{align*}
 \cF_1(\lambda,h):=&\left(\frac{h_q}{H_p+h_p}\right)_q-\left(\Gamma+\frac{1+h_q^2}{2(H_p+h_p)^2}\right)_p,\\
 \cF_2(\lambda,h):=&\tr_0h+(1-\p_q^2)^{-1}\tr_0\left(\frac{\left(1+h_q^2+(2g(H+h)-Q)(H_p+h_p)^2\right)(1+h_q^2)^{3/2}}{2\sigma(H_p+h_p)^2}-h\right)
\end{align*}
for $(\lambda,h)\in (2\max_{[p_0,0]}\Gamma,\infty)\times X,$
whereby $H=H(\cdot;\lambda)$ and $Q=Q(\lambda)$ are  given by \eqref{E:LFS} and \eqref{Q}, respectively.
Let us observe that the function $\cF$ is well-defined and   it depends real-analytically on its arguments, that is 
\begin{align}\label{BP0}
 \cF\in C^\omega((2\max_{[p_0,0]}\Gamma,\infty)\times X, Y).
\end{align}
Whence, the problem \eqref{PB'} is equivalent to the following abstract  equation
\begin{align}\label{BP}
 \cF(\lambda,h)=0\qquad\text{in $Y$,}
\end{align}
the laminar flow solutions of \eqref{PB'} corresponding to the trivial solutions of $\cF$
\begin{align}\label{BP1}
 \cF(\lambda,0)=0\qquad\text{for all $\lambda\in(2\max_{[p_0,0]}\Gamma,\infty).$}
\end{align}
We emphasize that  if $(\lambda,h)$ is a solution of \eqref{BP}, then the function $h+H(\cdot;\lambda)$  is a weak solution of \eqref{PB'}, when $Q=Q(\lambda)$, 
and it also satisfies the condition \eqref{PBC} if $h$ is sufficiently small.

In order to prove the existence of branches of solutions of \eqref{BP}  bifurcating from the laminar flows $h=0,$ we need to determine particular $\lambda$ for which  $\p_{h} \cF (\lambda,0)\in\kL(X,Y)$ 
is a Fredholm operator of index zero with a one-dimensional kernel. 
We first prove   that  $\p_{h}\cF(\lambda,0)$ is a Fredholm operator of index zero for every value of $\lambda\in(2\max_{[p_0,0]}\Gamma,\infty).$ 
To this end, given $\lambda\in(2\max_{[p_0,0]}\Gamma,\infty),$  we note that the Fr\' echet derivative $\p_{h} \cF (\lambda,0)$ is the linear operator $(L,T)\in\kL(X,Y)$ given by
\begin{equation}\label{L1}
\begin{aligned}
 Lw:=& \left(\frac{w_q}{H_p}\right)_q+\left(\frac{w_p}{H_p^3}\right)_p,\\
 Tw:=&\tr_0 w+(1-\p_q^2)^{-1} \tr_0 \left(\frac{gw-\lambda^{3/2}w_p}{\sigma}-w\right)
\end{aligned}\qquad\quad\text{for $w\in X.$}
\end{equation}

\begin{lemma}\label{L:2}
Given $\lambda\in(2\max_{[p_0,0]}\Gamma,\infty),$    the Fr\' echet  derivative  $\p_{h}\cF(\lambda,0)\in\kL(X,Y)$ is a Fredholm operator of index zero.
\end{lemma}
\begin{proof}
Given $w\in X,$ we note that
\[
(L,T)w=(L, \tr_0 )w+\left(0,(1-\p_q^2)^{-1} \tr_0\left(\frac{gw-\lambda^{3/2}w_p}{\sigma}-w\right)\right).
\]
Recalling that $(1-\p_q^2)^{-1}\in \kL(C^{\alpha}(\s), C^{2+\alpha}(\s))$, the operator
\[X\ni w\mapsto \left(0,(1-\p_q^2)^{-1} \tr_0\left(\frac{gw-\lambda^{3/2}w_p}{\sigma}-w\right)\right)\in Y\]
is  compact, so that our conclusion is immediate if  $(L,\tr_0):X\to Y$ is an isomorphism.
However, the latter property follows readily from the existence and uniqueness result  stated in \cite[Theorem 8.34]{GT01}.
\end{proof}

\paragraph{\bf The kernel of the Fr\'echet derivative}
We now identify certain $\lambda$ for which the Fr\' echet derivative $\p_{h}\cF(\lambda,0)=(L,T)$ has a one-dimensional kernel.
To this end, let $w\in X$  be a vector in the kernel of $(L,T)$ and define for each $k\in\N$ the Fourier coefficients
\[
w_k(p):=\la w(\cdot, p)|\cos(kn\cdot)\ra_{L_2}:=\int_0^{2\pi} w(q,p)\cos(knq)\, dq\qquad\text{for $p\in[p_0,0]$.}
\]
Clearly, we have   $w_k\in C^{1+\alpha}([p_0,0])$ for all $k\in\N.$
Given $\psi\in C^1_0((p_0,0)),$ we define the function  $\phi(q,p):=\psi(p)\cos(knq)$ for $(q,p)\in\0,$ and observe that
  $\phi\in C^1_0(\0).$ 
  Whence, in virtue of $ L w=0$,   integration by parts gives
\[
\int_{p_0}^0\left(\frac{w_k'}{H_p^3}\psi'+\frac{(kn)^2w_k}{H_p}\psi\right) dp=0.
\]
This relation being true for all $\psi\in C^1_0((p_0,0))$ and since   $H_p\in C^{1-}([p_0,0])=W^1_\infty((p_0,0))$, we conclude that $w_k\in H^2((p_0,0))$
is a strong solution of the equation
\[
\left(\frac{w_k'}{H_p^3}\right)'-\frac{(kn)^2w_k}{H_p}=0\qquad\text{in $L_2((p_0,0)).$}
\]
Moreover $w\in X $ implies that $w_k(p_0)=0$, while multiplying the relation $Tw=0$ by $\cos(knq)$ and making use of the symmetry of the operator $(1-\p_q^2)^{-1},$
that is
\begin{align}\label{BBB}
\la f|(1-\p_q^2)^{-1}g\ra_{L_2}=\la  (1-\p_q^2)^{-1}f|g\ra_{L_2}\qquad\text{for all $f,g\in C^\alpha(\s)$},
\end{align}
we determine a third relation
\[(g+\sigma (kn)^2)w_k(0)=\lambda^{3/2}w_k'(0).\]
Summarizing,  the Fourier coefficient $w_k\in H^2((p_0,0))$ solves the problem
\begin{equation}\label{E:m}
\left\{
\begin{array}{rlll}
  (a^3 w')'-\mu aw&=&0 &\text{in $L_2((p_0,0))$,}\\
  (g+\sigma\mu)w(0)&=&\lambda^{3/2}w'(0),\\
  w(p_0)&=&0,
  \end{array}\right.
\end{equation}
when $\mu=(kn)^2.$
Hereby, we use the shorthand $a:=1/H_p\in C^{1-}([p_0,0]).$
Thus, if we wish that $\p_{h}\cF(\lambda,0) $ has a one-dimensional kernel,   we need to impose conditions on $\lambda$
which guarantee that the system \eqref{E:m} has non-trivial solutions--which form a one-dimensional subspace of $H^2((p_0,0))$--for only one constant $\mu\in\{(kn)^2\,:\, k\in\N\}.$
This motivates us to study,  for each  $(\lambda,\mu)\in(2\max_{[p_0,0]}\Gamma,\infty)\times[0,\infty),$ the Sturm-Liouville operator 
$R_{\lambda,\mu}:H \to L_2((p_0,0))\times  \R , $ whereby $H:=\{w\in H^2((p_0,0))\,:\, w(p_0)=0\}$  and  
\begin{equation*}
 R_{\lambda,\mu}w:=
 \begin{pmatrix}
  (a^3 w')'-\mu aw\\
  (g+\sigma\mu)w(0)-\lambda^{3/2}w'(0)
 \end{pmatrix}\qquad\text{for $w\in H.$}
\end{equation*}
Therefore, given  $(\lambda,\mu)\in(2\max_{[p_0,0]}\Gamma,\infty)\times [0,\infty)$, we define  the functions $v_i\in C^{2-}([p_0,0])$, with
 $v_i:=v_i(\cdot;\lambda,\mu)$,  as being the solutions of  the initial value problems 
\begin{equation}\label{ERU}
\left\{\begin{array}{lll}
  (a^3 v_1')'-\mu av_1=0\qquad \text{in $L_2((p_0,0))$},\\[1ex]
  v_1(p_0)=0,\quad v_1'(p_0)=1,
 \end{array}
 \right.\hspace{1cm}
 \left\{\begin{array}{lll}
  (a^3v_2')'-\mu av_2=0\qquad \text{in $L_2((p_0,0))$},\\[1ex]
  v_2(0)=\lambda^{3/2},\quad v_2'(0)=g+\sigma\mu.
 \end{array}
 \right.
\end{equation}
These problems can be seen as system of first order linear ordinary differential equations, and therefore the existence and uniqueness  of  $v_i$ follows from the classical theory, cf. \cite{A83}.  

\begin{prop}\label{P:2}
 For every  $(\lambda,\mu)\in(2\max_{[p_0,0]}\Gamma,\infty)\times[0,\infty)$, the operator $R_{\lambda,\mu}$ is a Fredholm operator of index zero and its kernel is at most one-dimensional.
 Furthermore, the kernel of the operator $R_{\lambda,\mu}$ is non-trivial exactly when  the functions  $v_i$, $i=1,2,$ given by \eqref{ERU}, are linearly dependent.
 In this case we have $\ke R_{\lambda,\mu}=\spa\{v_1\}.$
\end{prop}
\begin{proof}
 Observe first that the operator $R_{\lambda,\mu}$ can be writen as a sum $R_{\lambda,\mu}=R_I+R_c$, with
  \[
 R_Iw:=
 \begin{pmatrix}
 (a^3 w')'-\mu aw\\
  -\lambda^{3/2}w'(0)
 \end{pmatrix} \qquad \text{and}\qquad
R_cw:=
 \begin{pmatrix}
 0\\
  (g+\sigma\mu) w(0)
 \end{pmatrix} 
 \]
 for all $w\in H,$  
 $R_c$ being a compact operator.
 Furthermore, if the equation $R_Iw=(f,A), $ with $ (f,A)\in L_2((p_0,0))\times  \R$, has a solution $w\in H,$ then
  \begin{equation}\label{VF}
  \int_{p_0}^0\left(a^3w'\varphi'+\mu aw\varphi\right)dp=-A\varphi(0)-\int_{p_0}^0 f\varphi\, dp
 \end{equation}
 for all $\varphi\in H_*:=\{w\in H^1((p_0,0))\,:\, w(p_0)=0\}$.
 Noticing that the right-hand side of \eqref{VF} defines a linear functional in $\kL(H_*,\R),$ and that the left-hand side corresponds to a bounded bilinear and coercive functional in $H_*\times H_*,$
the existence and uniqueness of a solution $w\in H_*$ follows from the Lax-Milgram theorem, cf. \cite[Theorem 5.8]{GT01}.
This solution is actually in $H$ and therefore $R_I$ is an isomorphism. 
This proves the Fredholm property of $R_{\lambda,\mu}$.

In order to see that the  kernel of  $R_{\lambda,\mu}$ is at most one-dimensional, we consider two solutions $w_1,w_2\in H^2((p_0,0)) $ of the equation $(a^3 w')'-\mu aw=0$ in $L_2((p_0,0)).$
Multiplying the   equation satisfied by $w_1$ with $ w_2$ and that satisfied by $w_2$ with $w_1,$ we obtain, after subtracting the new identities, that
\begin{equation}\label{BV}a^3(w_1w_2'-w_2w_1')=const. \qquad\text{in $[p_0,0]$}.\end{equation}
Thus, if additionally $w_1, w_2\in H,$ then the constant is zero and, in view of $a>0,$ $w_1$ and $w_2$ are linearly dependent.
Finally, it is not difficult to see that if the functions $v_1$ and $v_2$, given by \eqref{ERU}, are linearly dependent, then they both belong to $\ke R_{\lambda,\mu}.$
On the other hand, if  $0\neq v\in \ke R_{\lambda,\mu},$ using the relation \eqref{BV}, we get that $v$ is colinear with $v_1 $ and $v_2.$ This proves the claim.
 \end{proof}

Thus, we need to determine for which $(\lambda,\mu)$ the Wronskian 
\[
W(p;\lambda,\mu):=\left|
\begin{array}{lll}
 v_1&v_2\\
 v_1'&v_2'
\end{array}
\right|
\]
vanishes on the whole interval $[p_0,0].$ 
Recalling \eqref{BV}, the Wronskian vanishes on $[p_0,0]$ if and only if it vanishes at $p=0.$
Summarizing,  $  R_{\lambda,\mu} $  has a  one-dimensional kernel exactly when $(\lambda,\mu)$ is a solution of the equation $W(0;\lambda,\mu)=0.$
Taking into account that all the equations of \eqref{ERU} depend real-analytically on  the variable $(\lambda,\mu),$ we deduce that
the function   $W(0;\cdot,\cdot):(2\max_{[p_0,0]}\Gamma,\infty)\times  [0,\infty)\to\R$, defined by 
\begin{equation}\label{DEF}
W(0;\lambda,\mu):=\lambda^{3/2}v_1'(0;\lambda,\mu)-(g+\sigma\mu)v_1(0;\lambda,\mu),
\end{equation}
is real-analytic.
Determining the zeros of $W(0;\cdot,\cdot)$ when $\mu=0$ is rather easy. 
Indeed, for  $\mu=0,$ we can determine $v_1$ explicitly
\[
v_1(p;\lambda,0)=\int_{p_0}^p\frac{a^3(p_0)}{a^3(s)}\, ds,\qquad p\in[p_0,0].
\]
Consequently,  $W(0;\lambda,0)=0$ if and only if $\lambda$ solves the equation
\begin{equation}\label{QU}
 \frac{1}{g}=\int_{p_0}^0\frac{1}{a^3(p)}\, dp.
\end{equation}
The right-hand side of \eqref{QU} is a strictly decreasing function of $\lambda$,
\[
\int_{p_0}^0\frac{1}{a^3(p)}\, dp \, { \underset{\lambda\to\infty}{\longrightarrow } 0}\qquad\text{and}\qquad \int_{p_0}^0\frac{1}{a^3(p)}\, dp
 \, {\underset{\lambda\to 2\max_{[p_0,0]}\Gamma}{\longrightarrow} \infty}.
\]
Consequently, there exists a unique $\lambda_0\in(2\max_{[p_0,0]}\Gamma,\infty)$ which satisfies \eqref{QU}.
When $\mu>0,$ there are in general no explicit formula for $v_1,$ and the problem of determining the zeros of $W(0;\cdot,\cdot)$ is more intriguing.
However, the arguments used in \cite{CM13xx} can be   adapted  to our context to prove the following statement.

\begin{prop}\label{P:4}
Given $\lambda>\lambda_0,$ there exists a unique solution $\mu(\lambda)\in(0,\infty)$ of the equation $W(0;\lambda,\mu)=0.$
Moreover, we have $W(0;\lambda_0,\cdot)^{-1}\{0\}=\{0,\mu(\lambda_0)\}$   whereby $\mu(\lambda_0)=0$ if
\begin{equation}\label{d2}
   \int_{p_0}^0 a(p)\left(\int_{p_0}^p\frac{1}{a^3(s)}\, ds\right)^2\, dp<\frac{\sigma}{g^2}.
 \end{equation}
The function $\mu:[\lambda_0,\infty)\to[0,\infty)$ is real-analytic in $(\lambda_0,\infty)$, strictly increasing, and
\begin{equation}\label{lim}
 \lim_{\lambda\to\infty}\mu(\lambda)=\infty.
\end{equation}
\end{prop}
\begin{proof}
 The proof is similar to that of the Lemmas 4.3 - 4.7 in \cite{CM13xx}, the restriction  $\gamma\in L_\infty((p_0,0))$ leading only to minor modifications.
 Therefore, we omit it.
\end{proof}

In virtue of Proposition \ref{P:4}, there exists a smallest positive integer $n$ ($n=1$ if \eqref{d2} is satisfied) such that
for all $k\in\N\setminus\{0\},$ there exists a unique constant $ \lambda_k\in(\lambda_0,\infty)$ with the property that 
\begin{align}\label{LP}
 \mu(\lambda_k):=(kn)^2.
\end{align}
Because $\mu$ is strictly increasing and recalling \eqref{lim}, we deduce that $\lambda_k\nearrow\infty$. 
Summarizing, if $k\geq1,$ we have that $W(0;\lambda_k,(ln)^2)= 0$, with $l\in\N,$ if and only if  $l=k.$
Consequently, $R_{\lambda_k,(ln)^2}$ has a non-trivial kernel if and only if $l=k.$
We have thus shown that, for all $k\geq1$, the kernel of the Fr\'echet derivative $\p_h\cF(\lambda_k,0)$ is one-dimensional. 
More precisely, we have
\begin{equation}\label{KER}
\ke\p_h\cF(\lambda_k,0)=\spa\{w_* \},
\end{equation}
whereby $w_*(q,p):=v_1(p)\cos(knq)$ and $v_1$ is the solution of the first system of \eqref{ERU} when setting $\mu=(kn)^2.$
That $w_*$ is an element of  $X$ follows easily from  $v_1\in C^{2-}([p_0,0]).$  \smallskip

\paragraph{\bf The transversality condition}
In order to apply the   theorem on bifurcation from simple eigenvalues due to Crandall and Rabinowitz \cite{CR71} to the operator equation \eqref{BP} we still need to prove that
\begin{align}\label{TC}
 \p_{\lambda h}\cF(\lambda_k,0)[w_*]\notin \im \p_{h}\cF(\lambda_k,0).
\end{align}
To this end, we need to   characterize the range $\im \p_h\cF(\lambda_k,0).$
\begin{lemma}\label{L:R} Given   $k\in\N$ with $k\geq1$, the pair $(f, \varphi)\in Y$, with $f:=\p_q\phi_1+\p_p\phi_2$, belongs to $ \im \p_{h}\cF(\lambda_k,0)$ if and only if  we have
 \begin{equation}\label{RN}
  \int_\0 \phi_1w_{*q}+\phi_2w_{*p}\,d(q,p)-\int_{\s\times \{0\}} \phi_2  w_*\, dq-\sigma(1+(kn)^2)\int_{\s\times \{0\}} \varphi  w_*\, dq=0.
 \end{equation}
\end{lemma} 
\begin{proof}
Let us presuppose that  there exists $w\in X$ such that $(L,T)w=(f,\varphi).$ 
For every positive integer $m$, we define the function $\psi_m\in H^1_0((p_0,0))$ by the relation
\[
\psi_m(p):=\left\{
\begin{array}{llll}
 1&\text{for $p_0+1/m\leq x\leq -1/m,$}\\
 m(x-p_0)&\text{for $p_0\leq x\leq p_0+1/m,$}\\
 -mx&\text{for $-1/m\leq x\leq 0.$}
\end{array}
\right.
\]
Then $\psi_mw_*\in H^1_0(\0)=\overline{C^1_0(\0)}^{\|\cdot\|_{H^1}},$ and,  since $Lw=f$ in $Y_1,$ a density argument leads us to the following relation
\begin{align*}
\int_\0\frac{w_q}{H_p}\psi_m\p_qw_*+\frac{w_p}{H_p^3}\p_p(\psi_mw_*)\, d(q,p)= \int_\0 \phi_1\psi_m\p_qw_*+\phi_2\p_p(\psi_mw_*)\,d(q,p).
\end{align*}
Letting $m\to\infty$, it is easy to see that
\begin{align}
\int_\0\frac{w_qw_{*q}}{H_p} +\frac{w_p w_{*p}}{H_p^3} \, d(q,p)-\int_{\s\times \{0\}} \frac{w_p w_*}{H_p^3}\, dq= \int_\0 \phi_1w_{*q}+\phi_2w_{*p}\,d(q,p)-\int_{\s\times \{0\}} \phi_2  w_*\, dq.\label{R1a}
\end{align}
On the other hand, if we multiply  the relation $Tw=\varphi$ by $w_*$ and integrate it over a period, gives, after exploiting the relation 
\begin{align}\label{BBB1}
(1-\p_q^2)^{-1}(\tr_0 w_*)=\frac{\tr_0w_*}{1+(kn)^2}
\end{align}
and the  symmetry of the operator $(1-\p_q^2)^{-1}$, the following integral relation
\begin{align}
(g+\sigma(kn)^2)\int_{\s\times \{0\}} w w_*\, dq-\int_{\s\times \{0\}} \frac{w_p w_*}{H_p^3}  \, dq= \sigma(1+(kn)^2)\int_{\s\times \{0\}} \varphi  w_*\, dq.\label{R2a}
\end{align}
Subtracting  \eqref{R2a} from \eqref{R1a}, we find that
\begin{align}
 &\int_\0 \phi_1w_{*q}+\phi_2w_{*p}\,d(q,p)-\int_{\s\times \{0\}} \phi_2  w_*\, dq-\sigma(1+(kn)^2)\int_{\s\times \{0\}} \varphi  w_*\, dq\nonumber\\
 &=\int_\0\frac{w_qw_{*q}}{H_p} +\frac{w_pw_{*p}}{H_p^3}  \, d(q,p)-(g+\sigma(kn)^2)\int_{\s\times \{0\}} w w_*\, dq.\label{RR}
\end{align}
However, recalling that $(L,T)w_*=0$, similar  arguments to those presented above show that the right-hand side of \eqref{RR} is zero, and we obtain the desired relation \eqref{RN}.
To finish the proof, let us observe that the relation \eqref{RN} defines a closed subspace of $Y$ that has codimension one and contains the range $\im \p_h\cF(\lambda_k,0).$
Since the range also has codimension one, we conclude that every pair $(f,\varphi)$ that satisfies \eqref{RN} belongs $\im \p_h\cF(\lambda_k,0).$

\end{proof}

\begin{lemma}\label{L:TC} 
The transversality condition \eqref{TC} is satisfied for all $k\in\N$ with $k\geq 1$.
\end{lemma}
\begin{proof}
 Differentiating  \eqref{L1} with respect to $\lambda$    we obtain, in virtue of $\p_\lambda H_p=-1/(2H_p^3)$, that  
 \begin{align*}
  \p_{\lambda h}\cF(\lambda_k,0)[w_*]=\left(\left(\frac{H_pw_{*q}}{2 }\right)_q+\left(\frac{3w_{*p}}{2H_p}\right)_p,-\frac{3\lambda_k^{1/2}}{2\sigma}(1-\p_q^2)^{-1}\tr_0w_{*p}\right).
 \end{align*}
 We only need to check that the relation \eqref{RN} is not satisfied by   $\p_{\lambda h}\cF(\lambda_k,0)[w_*]\in Y$.
To this end, we set
\[
\phi_1:=\frac{H_pw_{*q}}{2 }, \qquad\phi_2:=\frac{3w_{*p}}{2H_p} , \qquad \varphi:=-\frac{3\lambda_k^{1/2}}{2\sigma}(1-\p_q^2)^{-1}\tr_0w_{*p},
\]
and, recalling  \eqref{BBB} and \eqref{BBB1},  we conclude that
\begin{align*}
 &\int_\0 \phi_1w_{*q}+\phi_2w_{*p}\,d(q,p)-\int_{\s\times \{0\}} \phi_2  w_*\, dq-\sigma(1+(kn)^2) \int_{\s\times \{0\}}\varphi  w_*\, dq\\
 &=\int_\0 \frac{H_pw_{*q}^2}{2 } +\frac{3w_{*p}^2}{2H_p} \,d(q,p)>0.
\end{align*}
\end{proof}

Gathering \eqref{BP0}, \eqref{BP1}, \eqref{KER}, Proposition \ref{P:4} and the Lemmas \ref{L:2} and \ref{L:TC}, the theorem on bifurcation from simple eigenvalues due to Crandall and Rabinowitz
\cite{CR71} yields the following result for the bifurcation problem \eqref{BP}.
\begin{thm}[Local bifurcation]\label{T:LB}
Let  $\gamma\in L_\infty((p_0,0))$ be given.
Then,  there exists a positive integer $n$ and, for each  $k\in\N\setminus\{0\} $, there exists   $\e_k>0$ and a real-analytic  curve 
\[\text{$(\ov\lambda_k,h_k):(\lambda_k-\e_k,\lambda_k+\e_k)\to  (2\max_{[p_0,0]}\Gamma,\infty)\times X,$  }\]
consisting only of solutions of the problem
 \eqref{BP}.
Moreover,   we have that
\[
\begin{aligned}
&\ov\lambda_k(s)=\lambda_k+O(s),\\
& h_k(s)=sw_*+O(s^2),
\end{aligned} \qquad\qquad \text{as $s\to0$,}
\]
whereby $w_*\in X$ is given by $w_*(q,p):=v_1(p)\cos(knq)$ and $v_1$ denotes the solution of the first system of \eqref{ERU} when  $\mu=(kn)^2.$
Moreover, in a neighborhood of $(\lambda_k,0),$ the solutions of \eqref{BP} are either trivial or are located on the local curve $(\ov\lambda_k,h_k)$.
If the condition \eqref{d2} is satisfied, then $n=1$.
\end{thm}

The points on the curves $(\ov\lambda_k,h_k), $ $k\geq1$ correspond to weak solutions of the problem \eqref{PB'} and \eqref{PBC}.
The next lemma shows, under an additional regularity assumption, that all  weak solutions $h\in C^{1+\alpha}(\ov\0)$ of \eqref{PB'} and \eqref{PBC} 
are in fact strong solutions (even classical solutions if $\gamma\in C([p_0,0])$). 
This additional regularity assumption is shown later on, cf.  Proposition \ref{P:1}, 
 to be a priori satisfied by the weak solutions $h\in C^{1+\alpha}(\ov\0)$ of \eqref{PB'} and \eqref{PBC} even when the vorticity function is merely integrable.
\begin{lemma}\label{L:Re} Assume that $h\in C^{1+ \alpha}(\ov\0)$ is a weak solution of \eqref{PB'} and \eqref{PBC} corresponding to a
 vorticity function $\gamma\in L_\infty((p_0,0)) $  (resp. $\gamma\in C([p_0,0])$)
Additionally, we assume that $h_q\in C^{1+\alpha}(\ov\0).$
Then, $h\in W^2_\infty(\0)$ (resp. $h\in C^{2}(\ov\0)$) and $h$ satisfies the first  equation of \eqref{PB} almost everywhere in $\0$. 
\end{lemma}
\begin{proof}
 Because $\p_p(h_q)\in C^\alpha(\ov\0),$ we deduce that $h_p$ is differentiable with respect to $q$ and that $\p_qh_p=\p_ph_q\in C^\alpha(\ov\0).$
Therefore, we have that
\[
\p_p\left(\Gamma+\frac{1+h_q^2}{2h_p^2}\right)=\p_q\left(\frac{h_q}{h_p}\right)\in C^\alpha(\ov\0)\qquad\text{and}
\qquad\p_q\left(\Gamma+\frac{1+h_q^2}{2h_p^2}\right)\in C^\alpha(\ov\0),
\]
the first relation being understood in the sense of distributions. 
Particularly, we deduce that
\[
 \Gamma+\frac{1+h_q^2}{2h_p^2}  \in C^{1-}(\ov\0).
\]
But, since $h_p$ satisfies \eqref{PBC} and it is also bounded, this implies $h_p\in C^{1-}(\ov\0).$
Summarizing, we have shown that $h\in C^{2-}(\ov\0)=W^2_\infty(\0).$
The same arguments lead us $h\in C^{2}(\ov\0)$ if $\gamma\in C([p_0,0]).$
The final part of the claim follows now directly from \eqref{PB1}, cf. \cite[Lemma 7.5]{GT01}. 
\end{proof}

\section{Regularity of weak solutions}\label{Sec:5}
We consider now an arbitrary non-laminar weak solution $h\in C^{1+\alpha}(\ov\0)$  of the water wave problem \eqref{PB'}, when requiring merely integrability of the vorticity function.
Assuming  that $h$ satisfies also the condition \eqref{PBC},  we establish additional regularity properties for this  weak solution.
The main result of this section is the following theorem.

 \begin{thm}[Regularity result]\label{MT2}
 Assume that $\gamma\in L_1((p_0,0))$ and let $\alpha\in(0,1)$ be given. 
Given a weak solution $h\in C^{1+\alpha}(\ov\0)$ of \eqref{PB'} that satisfies \eqref{PBC}, we have that $\p_q^m h\in C^{1+\alpha}(\ov\0)$ for all $m\in\N.$ 
Moreover, there exists a constant $L>1$
 with the property that
 \begin{equation}\label{E}
  \|\p_q^m h\|_{{1+\alpha}}\leq L^{m-2}(m-3)!
 \end{equation}
for all integers $m\geq3.$
\end{thm}

An immediate consequence of Theorem \ref{MT2} is the following corollary.

\begin{cor}\label{C:1} 
 Let $h$  satisfy the assumptions of Theorem \ref{MT2}.
 Then, the wave surface and all the other streamlines are real-analytic graphs. 
\end{cor}
\begin{proof}
In view of Theorem \ref{MT2}, the function $h(\cdot,p)$ is  real-analytic for all $p\in[p_0,0].$
Since the streamlines of the flow coincide with  the graphs $[q\mapsto h(q,p)-d]$, the conclusion is obvious.
\end{proof}

\begin{rem}\label{R:1} The result established in Theorem \ref{MT2} is true also for pure capillary water waves.
Indeed, neglecting gravity  corresponds to putting $g=0$ in \eqref{PB'}, modification which does not influence  the proof  of Theorem \ref{MT2}.
Particularly, the streamlines and the wave profile of capillary water waves with a merely integrable vorticity function are real-analytic.
This generalizes previous results  \cite{LW12x, DH12}.
\end{rem}

We  first prove   that the distributional derivative $\p_q^mh$, $m\in\N$ and $m\geq1,$ is a weak  solution of a linear elliptic equation satisfying certain nonlocal boundary conditions.
This property appears as a consequence of the   invariance of the problem \eqref{PB'} with respect to horizontal translations.

\begin{prop}\label{P:1}
Let $\gamma\in L_1((p_0,0))$ be given and assume that $h\in C^{1+\alpha}(\ov\0)$ is a weak solution of \eqref{PB'} which satisfies additionally \eqref{PBC}.
Given $m\in\N $ with $m\geq1,$ the derivative $\p_q^m h$ belongs to $ C^{1+\alpha}(\ov \0)$ and it is a weak solution\footnote{$\p_q^mh$ solves the first equation of 
\eqref{WH} in the weak sense  and the two boundary conditions pointwise.} of the elliptic boundary value problem
 \begin{equation}\label{WH}
\left\{
\begin{array}{rllll}
\left(\frac{1}{h_p}\p_q w\right)_q-\left(\frac{h_q}{h_p^2}\p_p w\right)_q-\left(\frac{h_q}{h_p^2}\p_q w\right)_p+\left(\frac{1+h^2_q}{h_p^3}\p_p w\right)_p&=&\left( f_m\right)_q+\left(g_m\right)_p&\text{in}&\0,\\[1ex]
w-(1-\p_q^2)^{-1}\tr_0(w+a_1w_q+a_2w_p)&=&\sum_{i=1}^5\varphi^i_m&\text{on} &p=0,\\[2ex]
w&=&0&\text{on}&p=p_0,
\end{array}
\right.
\end{equation}
whereby $f_m, g_m, \varphi_m\in C^\alpha(\ov\0)$   are given by 
\begin{align*}
f_m:=&\sum_{k=1}^{m-1}\begin{pmatrix}m-1\\k\end{pmatrix}\left[-\p_q^k\left(\frac{1}{h_p}\right)\p_q(\p_q^{m-k}h) +\p_q^k\left(\frac{h_q}{h_p^2}\right)\p_p(\p_q^{m-k}h)\right],\\
g_m:=&\sum_{k=1}^{m-1}\begin{pmatrix}m-1\\k\end{pmatrix}\left[ \p_q^k\left(\frac{h_q}{h_p^2}\right)\p_q(\p_q^{m-k}h) -\p_q^k\left(\frac{1+h^2_q}{h_p^3}\right)\p_p(\p_q^{m-k}h)\right],
\end{align*}
and $\varphi_m^i, a_i\in C^\alpha(\ov\0)$ are defined as
\begin{align*}
a_1:=&-\frac{5(1+h_q^2)^{3/2}h_q}{2\sigma h_p^2}-\frac{3(2gh-Q)(1+h_q^2)^{1/2}h_q}{2\sigma}, \qquad a_2:= \frac{(1+h_q^2)^{5/2}}{\sigma h_p^3},\\
\varphi_m^1:=&-\frac{1}{2\sigma}(1-\p_q^2)^{-1}\tr_0\left(\sum_{k=1}^{m-1} \begin{pmatrix}m\\k\end{pmatrix}(\p_q^k(1+h_q^2)^{5/2})\p_q^{m-k}\frac{1}{h_p^2}\right),\\
\varphi_m^2:=&-\frac{g}{\sigma}(1-\p_q^2)^{-1}\tr_0\left(\sum_{k=0}^{m-1} \begin{pmatrix}m\\k\end{pmatrix}(\p_q^k(1+h_q^2)^{3/2})\p_q^{m-k}h\right),\\
\varphi_m^3:=&\frac{1}{\sigma}(1-\p_q^2)^{-1}\tr_0\left((1+h_q^2)^{5/2}\sum_{k=0}^{m-2} \begin{pmatrix}m-1\\k\end{pmatrix}(\p_q^{k+1}h_p)\p_q^{m-k-1}\frac{1}{h_p^3}\right),\\
\varphi_m^4:=&-\frac{5}{2\sigma}(1-\p_q^2)^{-1}\tr_0\left(\frac{1}{h_p^2}\sum_{k=0}^{m-2} \begin{pmatrix}m-1\\k\end{pmatrix}(\p_q^{k+1}h_q)\p_q^{m-k-1}(h_q(1+h_q^2)^{3/2})\right),\\
\varphi_m^5:=&-\frac{3}{2\sigma}(1-\p_q^2)^{-1}\tr_0\left((2gh-Q)\sum_{k=0}^{m-2} \begin{pmatrix}m-1\\k\end{pmatrix}(\p_q^{k+1}h_q)\p_q^{m-k-1}(h_q(1+h_q^2)^{1/2})\right).\\
\end{align*}
\end{prop}

We denote in this section by  $C_i,$ $i\in\N,$  universal constants which are independent of $m$ and the function $h$ considered in Proposition \ref{P:1}.
Moreover, we use $K_i$, $i\in\N,$ to denote also constants that are independent of $m$, but may depend on $\|\p_q^lh\|_{1+\alpha}$ with $0\leq l\leq 2.$

\begin{proof}[Proof of Proposition \ref{P:1}]
 The proof follows by using the induction principle.
 We will only show that $\p_qh $ belongs to $ C^{1+\alpha}(\ov\0)$ and that it solves the system \eqref{WH} when $m=1$.
 The general induction step follows by using a similar argument as in this first induction  step (see e.g. the proof of \cite[Proposition 2.1]{EM13x}).
 
 To begin, we observe that for all $\e\in(0,1),$ the horizontal translation $ h_\e\in C^{1+\alpha}(\ov\0)$ defined by $h_\e(q,p):=h(q+\e,p)$ for $(q,p)\in\ov\0$
 is also a weak solution of \eqref{PB'} and  \eqref{PBC}.
 Subtracting the relations satisfied by $h_\e$ from those satisfied by $h$, we find that the function $u_\e:=(h_\e-h)/\e$ belongs to  $C^{1+\alpha}(\ov\0)$ and it
 is a weak solution of the
equation 
\begin{equation}\label{WH3}
\left(a_{11}^\e\p_q u^\e\right)_q+\left(a_{12}^\e\p_p u^\e\right)_q
+\left(a_{21}^\e\p_q u^\e\right)_p+\left(a_{22}^\e\p_p u^\e\right)_p=0\qquad\text{in $\0$,}
\end{equation}
whereby 
\begin{align*}
  a_{11}^\e= \frac{1}{h_{\e,p}},\quad a_{12}^\e=-\frac{h_q}{h_ph_{\e,p}},
 \quad a_{21}^\e= -\frac{h_q+ h_{\e,q}}{ 2h_{\e, p}^2}, \quad a_{22}^\e=  \frac{(h_p+h_{\e,p})(1+h_q^2)}{2h^2_ph_{\e,p}^2}.
\end{align*}
Because of \eqref{PBC}, the equation \eqref{WH3} is uniformly elliptic when $\e\in(0,1) $ is sufficiently small.
Furthermore, $u_\e$ also satisfies the boundary conditions
\begin{equation}\label{WH4}
\left\{
\begin{array}{rllll}
u_\e&=&(1-\p_q^2)^{-1}\tr_0\left(b_1 u_\e+b_2u_{\e,q}+b_3u_{\e,p}\right)&\text{on} &p=0,\\
u^\e&=&0&\text{on}&p=p_0,
\end{array}
\right.
\end{equation}
with $b_i$ given by
\begin{align*}
 b_1:=&1-\frac{g(1+h_q^2)^{3/2}}{\sigma},\qquad b_3:=\frac{(h_p+h_{\e,p})(1+h_q^2)^{5/2}}{2\sigma h_p^2 H_p^2},\\
 b_2:=&-\frac{(h_q+h_{\e,q})\sum_{i=0}^4(1+h_q^2)^i(1+h_{\e,q}^2)^{4-i}}{2\sigma h_{\e,p}^2((1+h_q^2)^{5/2}+(1+h_{\e,q}^2)^{5/2})}\\
 &-\frac{(2gh_\e-Q)(h_q+h_{\e,q})\sum_{i=0}^2(1+h_q^2)^i(1+h_{\e,q}^2)^{2-i}}{2\sigma  ((1+h_q^2)^{3/2}+(1+h_{\e,q}^2)^{3/2})}.
\end{align*}
In view of \eqref{WH3} and \eqref{WH4}, we may use Schauder estimates for weak solutions of Dirichlet problems, cf. \cite[Theorem 8.33]{GT01} to conclude that there exists a positive constant $K_0,$ which is independent of $\e,$ such that
\begin{equation}\label{SE1}
 \|u_\e\|_{1+\alpha}\leq K_0\left(\|u_\e\|_0+\|(1-\p_q^2)^{-1}\tr_0\left(b_1 u_\e+b_2u_{\e,q}+b_3u_{\e,p}\right)\|_{1+\alpha}\right)
\end{equation}
for sufficiently small $\e.$
We now prove that $\|u_\e\|_{1+\alpha}$ may be bounded from above by a constant which is independent of $\e.$
Indeed, the mean value theorem implies that
\begin{equation}\label{SE2}
\|u_\e\|_0\leq \|h_q\|_0.
\end{equation}
On the other hand, taking into account that  $(1-\p_q^2)^{-1}\in\kL(C^{\alpha/2}(\s),C^{2+\alpha/2}(\s))$ and using the algebra property of $C^{\alpha/2}(\s),$ we get
\begin{align}
 &\|(1-\p_q^2)^{-1}\tr_0\left(b_1 u_\e+b_2u_{\e,q}+b_3u_{\e,p}\right)\|_{1+\alpha}\nonumber\\
 &\leq C_0\|(1-\p_q^2)^{-1}\tr_0\left(b_1 u_\e+b_2u_{\e,q}+b_3u_{\e,p}\right)\|_{2+\alpha/2}\nonumber\\
 &\leq C_0\|\tr_0\left(b_1 u_\e+b_2u_{\e,q}+b_3u_{\e,p}\right)\|_{\alpha/2}\nonumber\\
 &\leq K_1\|\tr_0 u_\e\|_{1+\alpha/2},\label{SE3}
\end{align}
with $C_0$ and $K_1$ independent of $\e.$
The well-known interpolation property of the H\"older spaces  
\[ (C(\s),C^{1+\alpha}(\s))_{\theta,\infty}=C^{1+\alpha/2}(\s)\qquad\text{if $\theta=\frac{2+\alpha}{2(1+\alpha)}$,}\]
 cf. e.g. \cite{L95},  implies, via Young's inequality, that
 \begin{equation}\label{E:Int}
  \|\tr_0u\|_{1+\alpha/2}\leq C_1\|\tr_0u\|_0^{1-\theta}\|\tr_0u\|_{1+\theta}^\theta\leq\delta\|u\|_{1+\alpha}+C(\delta)\|u\|_{0}
 \end{equation}
 for all $\delta>0 $ and all $u\in C^{1+\alpha}(\ov\0),$ the constant $C(\delta)$ being positive.
 Particularly, if we choose $\delta:=(2K_0K_1)^{-1},$ the relations \eqref{SE1}-\eqref{E:Int} yield that
 \begin{equation}\label{DE}
 \|u_\e\|_{1+\alpha}\leq 2K_0\|h_q\|_0(1+K_1C(\delta))
 \end{equation}
  for all sufficiently small   $\e$,  the right-hand side of \eqref{DE} being independent of $\e.$
  Since $u_\e$ converges pointwise to $h_q$, we find, by using \eqref{DE},  a subsequence of $(u_{\e_k})_k$  which  converges to $h_q$ in $C^1(\ov\0).$
 The uniform bound \eqref{DE} implies that in fact $h\in C^{1+\alpha}(\ov\0).$
 Finally, passing to the limit $k\to\infty$ in \eqref{WH3} and \eqref{WH4} we recover, in view of $h_q\in C^{1+\alpha}(\ov\0) $,
 the relations \eqref{WH} with $m=1.$
 \end{proof}

The following lemma will be one of the main tools when estimating the norm of the solution $\p_q^mh $ of \eqref{WH}.
 \begin{lemma}\label{L:A1}  Let $n_0, N_0\in\N $ satisfy $2\leq n_0\leq N_0$,  and assume that $\p_q^n u_i\in C^\alpha(\ov \0)$ for all $0\leq n\leq N_0$ and $1\leq i\leq 5.$
 If  there exists a constant $L\geq 1$  and a  real number $r\in[0,n_0]$ such that $\|\p_q^n u_i\|_\alpha\leq L^{n-r}(n-n_0)! $
for all $n_0\leq n\leq N_0,$
then we find a constant $C_0=C_0(n_0)>1$ with the property that
\begin{align}\label{E:a1}
& \|\p_q^n(u_1u_2)\|_\alpha\leq C_0\left(1+\sum_{i=1}^2\sum_{l=0}^{n_0-1}\|\p_q^l u_i\|_\alpha\right)^{2} L^{n-r}(n-n_0)!,\\
& \|\p_q^n(u_1u_2u_3)\|_\alpha\leq C_0\left(1+\sum_{i=1}^3\sum_{l=0}^{n_0-1}\|\p_q^l u_i\|_\alpha\right)^{6} L^{n-r}(n-n_0)!,\label{E:a2}\\
& \|\p_q^n(u_1u_2u_3u_4u_5)\|_\alpha\leq C_0\left(1+\sum_{i=1}^5\sum_{l=0}^{n_0-1}\|\p_q^l u_i\|_\alpha\right)^{14} L^{n-r}(n-n_0)!\label{E:a3}
\end{align}
for all $n_0\leq n\leq N_0.$
\end{lemma}
\begin{rem}\label{R:C0} We will use the assertions of the  Lemma \ref{L:A1} several times in this paper with $n_0\in\{2,3,5\}.$
Since the constant $C_0$ depends only on $n_0$,  it is useful to define
\begin{equation}\label{C0}
 C_0:=\max\{C_0(2), C_0(3), C_0(5)\}.
\end{equation}
It is   important to stress that the quantities on the right-hand side of \eqref{E:a1}-\eqref{E:a3} contain only derivatives of $u_i$ which have order less than $n_0.$
\end{rem}

\begin{proof}[Proof of Lemma \ref{L:A1}]
Leibniz's rule implies that for all $ n_0\leq n\leq N_0$ we have
\begin{equation}\label{MS}
 \p_q^n(u_1u_2)=\sum_{k=0}^{n} \begin{pmatrix}n\\k\end{pmatrix}(\p_q^{k}u_1)\p_q^{n-k}u_2.
 \end{equation}
Taking into account that $\|u_1u_2\|_\alpha\leq \|u_1\|_\alpha\|u_2\|_\alpha$, we find  for all $n_0\leq n\leq \min\{2n_0-1,N_0\}$ that
 \begin{align}
  \left\|\p_q^n(u_1u_2)\right\|_\alpha
  \leq &n_0!\left(\sum_{k=0}^{n-n_0}+\sum_{k= n-n_0+1}^{n_0-1} +\sum_{k=n_0 }^n \right)\|\p_q^{k}u_1\|_\alpha\|\p_q^{n-k}u_2\|_\alpha\nonumber\\
    \leq &n_0! \left(\sum_{k=0}^{n-n_0}\|\p_q^{k}u_1\|_\alpha L^{n-k-r}(n-k-n_0)! \right)+n_0! \left(\sum_{k=0}^{n_0-1}\|\p_q^{k}u_1\|_\alpha \right)^2\nonumber
 \end{align}
 \begin{align}
       &+n_0!\sum_{k=n_0}^n L^{k-r}(k-n_0)!\|\p_q^{n-k}u_2\|_\alpha \nonumber\\
       \leq &n_0!\left(1+\sum_{k=0}^{n_0-1}\|\p_q^{k}u_1\|_\alpha \right)^2L^{n-r}(n-n_0)!.\label{sum1}
 \end{align}
On the other hand, if $N_0\geq 2n_0 $ and  $ 2n_0\leq n\leq N_0,$ then we split the sum \eqref{MS} as follows
\begin{align}\label{SMMM}
 \p_q^n(u_1u_2)=\left(\sum_{k=0}^{ n_0-1 }+\sum_{k= n_0 }^{  n-n_0 }+\sum_{k=  n-n_0+1}^{n}\right) \begin{pmatrix}n\\k\end{pmatrix}(\p_q^{k}u_1)\p_q^{n-k}u_2,
 \end{align}
 and obtain from the hypothesis that
 \begin{align}
 \left\| \sum_{k=0}^{ n_0-1 } \begin{pmatrix}n\\k\end{pmatrix}(\p_q^{k}u_1)\p_q^{n-k}u_2\right\|_\alpha
 \leq&  \sum_{k=0}^{ n_0-1 } \begin{pmatrix}n\\k\end{pmatrix}L^{n-k-r}(n-k-n_0)!\|\p_q^{k}u_1\|_\alpha\nonumber\\
 \leq & L^{n-r}(n-n_0)!\left(\sum_{k=0}^{n_0-1}\|\p_q^{k}u_1\|_\alpha \right)  \frac{n^{n_0}}{(n-2n_0+2)^{n_0}} \nonumber\\
 \leq &C_1\left(\sum_{k=0}^{n_0-1}\|\p_q^{k}u_1\|_\alpha \right)L^{n-r}(n-n_0)!.\label{sum2}
 \end{align}
 The third sum of \eqref{SMMM} can be estimated by the same expression \eqref{sum2}.
 Finally, the second term of \eqref{SMMM} is estimated as follows
  \begin{align}
 \left\| \sum_{k= n_0 }^{  n-n_0 } \begin{pmatrix}n\\k\end{pmatrix}(\p_q^{k}u_1)\p_q^{n-k}u_2\right\|_\alpha\leq&  \sum_{k= n_0 }^{  n-n_0 } \begin{pmatrix}n\\k\end{pmatrix}L^{k-r}(k-n_0)!L^{n-k-r}(n-k-n_0)!\nonumber\\
 \leq&  L^{n-r}(n-n_0)!\sum_{k= n_0 }^{  n-n_0 } \frac{n^{n_0}}{(n-k-n_0+1)^{n_0}(k-n_0+1)^{n_0}}\nonumber\\
 \leq & L^{n-r}(n-n_0)! \sum_{k= 1}^{  n-2n_0+1 } \frac{n^{n_0}}{(n-2n_0+2-k)^{n_0}k^{n_0}} \nonumber\\
 \leq &C_2 L^{n-r}(n-n_0)!,\label{sum3}
 \end{align}
 since, using the inequality $n_0(n-2n_0+2)\geq n$ for $n\geq 2n_0,$ we have
 \begin{align}\label{arg}
  \sum_{k= 1}^{n-2n_0+1 } \frac{n^{n_0}}{(n-2n_0+2-k)^{n_0}k^{n_0}}\leq n_0^{n_0}\sum_{k= 1}^{n-2n_0+1 } \frac{(n-2n_0+2)^{n_0}}{(n-2n_0+2-k)^{n_0}k^{n_0}}\leq2(2n_0)^{n_0}\sum_{k=1}^\infty \frac{1}{k^{n_0}},
 \end{align}
 the last series being finite as $n_0\geq2.$
 Gathering \eqref{sum1}, \eqref{sum2}, and \eqref{sum3}, we have established \eqref{E:a1}.
We  next apply the estimate \eqref{E:a1} to the functions
\[
v_1:=\frac{u_1u_2}{C_0\left(1+\sum_{i=1}^2\sum_{l=0}^{n_0-1}\|\p_q^l u_i\|_\alpha\right)^{2}}\qquad\text{and}\qquad v_2:=u_3
\]
 and obtain \eqref{E:a2}.
The last claim \eqref{E:a3} follows by applying \eqref{E:a1} to the functions
\[
w_1:=\frac{u_1u_2u_3 }{C_0\left(1+\sum_{i=1}^3\sum_{l=0}^{n_0-1}\|\p_q^l u_i\|_\alpha\right)^{6}}\qquad\text{and}\qquad w_2:=\frac{u_4u_5}{C_0\left(1+\sum_{i=4}^5\sum_{l=0}^{n_0-1}\|\p_q^l u_i\|_\alpha\right)^{2}}.
\]
\end{proof}

Because the functions $f_m, g_m,$ and $\varphi_m^i$ contain derivatives of $1/h_p$, when estimating their norms  we make use of the following result.
 \begin{lemma}\label{L:A2}  Assume that $\p_q^n u\in C^\alpha(\ov \0)$ for all $0\leq n\leq N$, with $N\geq 3,$ and let $C_0$ be the constant defined by \eqref{C0}.
If  there exists a constant  $L$ with
\begin{equation}\label{E:CC}
L\geq \|\p_q^2(1/u)\|_\alpha^{2}+\|\p_q^3(1/u)\|_\alpha^{2/3}+C_0^2\left(1+\sum_{l=0}^1(2\|\p_q^l (1/u)\|_\alpha +\|\p_q^{1+l}u\|_\alpha)\right)^{12}
\end{equation}
and $\|\p_q^n u\|_\alpha\leq L^{n-2}(n-3)!$
for all $3\leq n\leq N$, and if  $\inf_{\0} u>0$,
then we have
\begin{equation}\label{E:4}
 \|\p_q^n(1/u)\|_\alpha\leq L^{n-3/2}(n-3)!\qquad\text{for all $3\leq n\leq N.$}
\end{equation}
\end{lemma}
\begin{proof}
 In view of \eqref{E:CC}, it is clear that \eqref{E:4} is satisfied when $n=3.$
 So, let us assume that $N\geq 4$ and that \eqref{E:4} is satisfied for all $3\leq n\leq m-1$, whereby $m\leq N $ is arbitrarily chosen.
 We only need  to prove that \eqref{E:CC} holds for  $m$ too.
 Therefore, we   write
\[
 \p_q^m(1/u)=\p_q^{m-1}(u_1u_2u_3),
 \]
 whereby $u_1:=-\p_qu$ and $u_2=u_3:=1/u.$
 Our hypothesis, the induction assumption, the relation \eqref{E:CC}, and the fact that $L>1$, yield that 
 \begin{align*}
 &\|\p_q^n u_1\|_\alpha=\|\p_q^{n+1}u \|_\alpha\leq L^{n-1}(n-2)!,\\
 &\|\p_q^nu_2\|_\alpha=\|\p_q^{n}(1/u)\|_\alpha\leq L^{n-3/2}(n-2)!\leq L^{n-1}(n-2)!
 \end{align*}
 for all $2\leq n\leq m-1.$
Therefore, we may use the estimate \eqref{E:a2} of Lemma  \ref{L:1} (with $r=1$, $n_0=2$, and $N_0=m-1$)   to obtain, in view of \eqref{E:CC}, that
 \begin{align*}
  \|\p_q^m(1/u)\|_\alpha&=\|\p_q^{m-1}(u_1u_2u_3)\|_\alpha\leq C_0\left(1+\sum_{l=0}^1(2\|\p_q^l (1/u)\|_\alpha +\|\p_q^{1+l}u\|_\alpha)\right)^{6}L^{m-2}(m-3)!\\
  &\leq L^{m-3/2}(m-3)!,
 \end{align*}
which is the desired estimate.
\end{proof}

 In the proof of Theorem \ref{MT2},  we also need to estimate, when considering the boundary terms $ \varphi_m^i,$ expressions containing derivatives of   $(1+h_q^2)^{1/2}$.
 To this end, we need, additionally to the Lemma \ref{L:A1} and \ref{L:A2}, the following result.
  
 \begin{lemma}\label{L:A3}  Assume that $\p_q^n u\in C^\alpha(\ov \0)$ for all $0\leq n\leq N$, with $N\geq 3,$ and let $C_0$ be the constant given by \eqref{C0}.
If  there exists a constant  $L$ satisfying
\begin{align}\label{E:CCC}
L\geq& \|\p_q^3((1+u^2)^{-1/2}) \|_\alpha^{4/5}+\|\p_q^2((1+u^2)^{-1/2}) \|_\alpha^{4}+\|\p_q^2u\|_\alpha+C_0^4\left(1+2\sum_{l=0}^2\|\p_q^l u \|_\alpha\right)^8\nonumber\\
&+C_0^4\left(1+3\sum_{l=0}^1\|\p_q^l ((1+u^2)^{-1/2})\|_\alpha+\sum_{l=0}^1\|\p_q^l u\|_\alpha+\sum_{l=0}^1\|\p_q^{l+1} u\|_\alpha\right)^{56}\nonumber\\
&+C_0^4\left(1+ \sum_{l=0}^2\|\p_q^l (1+u^2)\|_\alpha+ \sum_{l=0}^2\|\p_q^l( (1+u^2)^{-1/2})\|_\alpha\right)^{8},
\end{align}
such that $\|\p_q^n u\|_\alpha\leq L^{n-2}(n-3)!$
for all $3\leq n\leq N$,
then we have
\begin{equation}\label{E:5}
 \|\p_q^n((1+u^2)^{1/2})\|_\alpha\leq L^{n-3/2}(n-3)!\qquad\text{for all $3\leq n\leq N.$}
\end{equation}
\end{lemma}
 \begin{proof}
  We first prove  that 
  \begin{equation}\label{E:6}
 \|\p_q^n((1+u^2)^{-1/2})\|_\alpha\leq L^{n-7/4}(n-3)!\qquad\text{for all $3\leq n\leq N.$}
\end{equation}
With our choice of $L$, it is clear that \eqref{E:6} is satisfied when $n=3.$
Let us now presuppose that $N\geq 4$ and that \eqref{E:6} is true for all $3\leq n\leq m-1$, whereby $m$ satisfies $4\leq m\leq N.$
It suffices to prove that \eqref{E:6}   holds true for $n=m.$ 
To this end, we observe that
\[
\p_q^m((1+u^2)^{-1/2})=-\p_q^{m-1}\left( ((1+u^2)^{-1/2})^3u u_q\right),
\]
and that \eqref{E:CCC} together with the induction assumption imply
\begin{align*}
 &\|\p_q^n((1+u^2)^{-1/2})\|_\alpha \leq  L^{n-7/4}(n-2)!\leq L^{n-1}(n-2)!,\\
 &\|\p_q^nu\|_\alpha  \leq L^{n-1}(n-2)!,\\
  &\|\p_q^nu_q\|_\alpha \leq\|\p_q^{n+1}u\|_\alpha\leq  L^{n-1}(n-2)!
\end{align*}
for all $2\leq n\leq m-1.$
These inequalities allow us to use the estimate \eqref{E:a3} of Lemma \ref{L:A1} (with $r=1$, $n_0=2,$ and $N_0=m-1$) in order to obtain   that
\begin{align*}
\|\p_q^m((1+u^2)^{-1/2})\|_\alpha=&\|\p_q^{m-1}( ((1+u^2)^{-1/2})^3u u_q)\|_\alpha\\
\leq& C_0\left(1+3\sum_{l=0}^1\|\p_q^l ((1+u^2)^{-1/2})\|_\alpha+\sum_{l=0}^1\|\p_q^l u\|_\alpha+\sum_{l=0}^1\|\p_q^l u_q\|_\alpha\right)^{14}\\
&\times L^{m-2}(m-3)!\\
\leq& L^{m-7/4}(m-3)!,
\end{align*}
when $L$ satisfies \eqref{E:CCC}. This proves \eqref{E:6}.

In order to prove \eqref{E:5}, we observe that our hypothesis together with \eqref{E:CCC} and the estimate \eqref{E:a1} of the Lemma \ref{L:A1} (with $r=2$, $n_0=3$, and $N_0=N$) imply  that
\begin{align*}
\|\p_q^n(1+u^2)\|_\alpha=&\|\p_q^n u^2 \|_\alpha\leq C_0\left(1+2\sum_{l=0}^2\|\p_q^l u \|_\alpha\right)^2 L^{n-2}(n-3)!\leq L^{n-7/4}(n-3)!
\end{align*}
for all $3\leq n\leq N$.
Whence, invoking \eqref{E:6} and the estimate \eqref{E:a1} of Lemma \ref{L:A1} (with $r=7/4$, $n_0=3$, and $N_0=N$), we deduce that
\begin{align*}
 \|\p_q^n((1+u^2)^{1/2})\|_\alpha=& \|\p_q^n((1+u^2)(1+u^2)^{-1/2})\|_\alpha\\
 \leq& C_0\left(1+ \sum_{l=0}^2\|\p_q^l (1+u^2)\|_\alpha+ \sum_{l=0}^2\|\p_q^l (1+u^2)^{-1/2}\|_\alpha\right)^2 L^{n-7/4}(n-3)!\\
 \leq& L^{n-3/2}(n-3)!
\end{align*}
for all $3\leq n\leq N,$ the  last inequality being a consequence of   our choice for $L$. 
This is the desired claim.
 \end{proof}
 Finally, we come to the proof of our second main result stated in Theorem \ref{MT2}.
 \begin{proof}[Proof of Theorem \ref{MT2}]
The proof uses the induction principle.
To this end, let $C_0$ be the constant defined by \eqref{C0}.
We  first pick a positive constant $L$ which satisfies 
\begin{align}
L\geq &\|\p_q^2(1/h_p)\|_\alpha^{2}+\|\p_q^3(1/h_p)\|_\alpha^{2/3}+C_0^2\left(1+\sum_{l=0}^1(2\|\p_q^l (1/h_p)\|_\alpha +\|\p_q^{1+l}h_p\|_\alpha)\right)^{12}\nonumber\\
&+\|\p_q^3((1+h_q^2)^{-1/2}) \|_\alpha^{4/5}+\|\p_q^2((1+h_q^2)^{-1/2}) \|_\alpha^4+\|\p_q^2h_q\|_\alpha+C_0^4\left(1+2\sum_{l=0}^2\|\p_q^l h_q \|_\alpha\right)^8\nonumber\\
&+C_0^4\left(1+3\sum_{l=0}^1\|\p_q^l ((1+h_q^2)^{-1/2})\|_\alpha+\sum_{l=0}^1\|\p_q^l h_q\|_\alpha+\sum_{l=0}^1\|\p_q^{l+1} h_q\|_\alpha\right)^{56}\nonumber\\
&+C_0^4\left(1+ \sum_{l=0}^2\|\p_q^l (1+h_q^2)\|_\alpha+ \sum_{l=0}^2\|\p_q^l ((1+h_q^2)^{-1/2})\|_\alpha\right)^{8} +\|\p_q^2 h_q^2 \|_\alpha^4+\sum_{l=0}^6\|\p_q^l h\|_{1+\alpha},\label{E:C}
\end{align}
and observe that  $L\geq1.$ 
Moreover, this choice ensures that the estimate \eqref{E} is satisfied, for this fixed $L$, when  $3\leq m\leq 6.$ 
We  next  assume that \eqref{E} is true for all $3\leq n\leq m-1$ whereby $m\geq7,$ and are left  to prove, possibly under some additional constraints on $L$ (see \eqref{E:constr}), that \eqref{E} holds also for $m.$
Recalling the Proposition \ref{P:1}, we know that $\p_q^mh$ is the solution of the elliptic boundary value problem \eqref{WH}.
Proceeding as in the proof of Proposition \ref{P:1}, the Schauder estimate \cite[Theorem 8.33]{GT01} together with the  inequality \eqref{E:Int} show that   the solution $\p_q^mh$ of \eqref{WH} can be estimated as follows 
 \begin{equation}\label{IND}
  \|\p_q^mh\|_{1+\alpha}\leq K_2\left(\|\p_q^mh\|_0+\|f_m\|_\alpha+\|g_m\|_\alpha+\sum_{i=1}^5\|\varphi_m^i\|_{1+\alpha}\right).
 \end{equation}
 
Whence, we are left to prove that the right-hand side of \eqref{IND} can be bounded from above by $L^{m-2}(m-3)!.$
To establish this property, we notice that the induction assumption implies that
\begin{align}\label{S1}
 \max\{\|\p_q^nh_q\|_\alpha,\|\p_q^nh_p\|_\alpha\}\leq \|\p_q^nh\|_{1+\alpha}\leq L^{n-2}(n-3)! \qquad\text{for $3\leq n\leq m-1.$}
\end{align}
It is now immediate to observe, due to \eqref{E:C}, that $L\geq\max\{ \|\p_q^2h_q\|_\alpha^2,\|\p_q^2h_p\|_\alpha^2\}$, so that we also have
\begin{align}\label{S3}
 \max\{\|\p_q^nh_q\|_\alpha, \|\p_q^nh_p\|_\alpha\}\leq L^{n-3/2}(n-2)! \qquad\text{for $2\leq n\leq m-1.$}
\end{align}
On the other hand, the relation satisfied by $L$, the estimate \eqref{S1}, and the assumption \eqref{PBC} on $h$, guarantee   that the function $u:=h_p$ satisfies the assumptions of the Lemma \ref{L:A2} (with $N=m-1$).
Therefore, we conclude that
 \begin{align}\label{S2}
 \|\p_q^n(1/h_p)\|_\alpha\leq  L^{n-3/2}(n-3)! \qquad\text{for all $3\leq n\leq m-1,$}
\end{align}
which implies, in view of $L\geq \|\p_q^2(1/h_p)\|_\alpha^2$ that
  \begin{align}\label{S2'}
 \|\p_q^n(1/h_p)\|_\alpha\leq  L^{n-3/2}(n-2)! \qquad\text{for all $2\leq n\leq m-1.$}
\end{align}

With these preparations, we start and estimate the right-hand side of \eqref{IND}.
We begin by observing that
\begin{align}\label{F1}
\|\p_q^m h\|_0\leq\|\p_q^{m-1} h\|_\alpha\leq L^{m-3}(m-4)!.
\end{align}

Next we consider the expressions $\|f_m\|_\alpha$ and $\|g_m\|_\alpha.$
The arguments used for bounding these quantities are quite similar, and we will present them in detail only when estimating a representative term of $f_m.$
Indeed, recalling  \eqref{S3},  \eqref{S2'}, and the estimates \eqref{E:a1}-\eqref{E:a3} of  Lemma \ref{L:A1} (with $r=3/2,$ $n_0=2$, and $N_0=m-1$), we conclude that  there exists a constant $K_3>1$ such that
\begin{align}\label{G:1}
 \max\left\{\left\|\p_q^n\left(\frac{1}{h_p}\right)\right\|_\alpha,
\left\| \p_q^n\left(\frac{h_q}{h_p^2}\right)\right\|_\alpha, \left\|\p_q^n\left(\frac{1+h^2_q}{h_p^3}\right)\right\|_\alpha\right\}\leq K_3L^{n-3/2}(n-2)!
\end{align}
for all $2\leq n\leq m-1.$
With this observation at hand, it is not difficult to see that all the terms defining $f_m$ and $g_m$ can
 be estimated by using the same arguments as when dealing with the following representative term of $f_m$
\[
 \sum_{k=1}^{m-1}\begin{pmatrix}m-1\\k\end{pmatrix}\p_q^k\left(\frac{1}{h_p}\right)\p_q(\p_q^{m-k}h)=\left(\sum_{k=1}^1+\sum_{k=2}^{m-3}+\sum_{k=m-2}^{m-1}\right)\begin{pmatrix}m-1\\k\end{pmatrix}\p_q^k\left(\frac{1}{h_p}\right)\p_q(\p_q^{m-k}h).
\]
In view of \eqref{G:1} and of the induction assumption, we get  
\begin{align*}
 \left\|\left(\sum_{k=1}^1 +\sum_{k=m-2}^{m-1}\right)\begin{pmatrix}m-1\\k\end{pmatrix}\p_q^k\left(\frac{1}{h_p}\right)\p_q(\p_q^{m-k}h)\right\|_\alpha\leq K_4L^{m-5/2}(m-3)!.
\end{align*}
On the other hand, using additionally the relations \eqref{arg} and \eqref{S1}, we have
\begin{align*}
 \left\|\sum_{k= 2}^{m-3} \begin{pmatrix}m-1\\k\end{pmatrix}\p_q^k\left(\frac{1}{h_p}\right)\p_q(\p_q^{m-k}h)\right\|_\alpha
 \leq& \sum_{k= 2}^{m-3} \begin{pmatrix}m-1\\k\end{pmatrix}\left\|\p_q^k\frac{1}{h_p}\right\|_\alpha\left\|\p_q^{m-k}h\right\|_{1+\alpha}\\
 \leq &K_5L^{m-7/2}\sum_{k= 2}^{m-3}  \begin{pmatrix}m-1\\k\end{pmatrix}(k-2)!(m-k-3)!\\
  \leq &K_5L^{m-7/2}(m-3)!\sum_{k= 2}^{m-3}  \frac{(m-1)^2}{(m-k-2)^2(k-1)^2}\\
  \leq &K_5L^{m-7/2}(m-3)!.
\end{align*}
Proceeding in the same way with the remaining terms of $f_m$ and $g_m$, we end up with
\begin{align}\label{F2}
 \|f_m\|_\alpha+\|g_m\|_\alpha\leq K_6L^{m-5/2}(m-3)!.
\end{align}

We are left to estimate  the terms $\|\varphi_m^i\|_{1+\alpha} $, $1\leq i\leq 5$.
To this end, we observe that our choice of the constant $L$ and the induction assumption $\|\p_q^nh\|_{1+\alpha}\leq L^{n-2}(n-3)!$ for $3\leq n\leq m-1$ yield, via Lemma \ref{L:A3}, that
\begin{align}\label{S4}
\|\p_q^n((1+h_q^2)^{1/2})\|_\alpha\leq L^{n-3/2}(n-3)!\qquad\text{for all $3\leq n\leq m-1.$} 
\end{align}
Since \eqref{S1} and the induction assumption imply  
\begin{align}\label{S1'}
 \max\{ \|\p_q^nh_q\|_\alpha, \|\p_q^nh_p\|_\alpha\}\leq L^{n-3/2}(n-3)! \qquad\text{for $3\leq n\leq m-1,$}
\end{align}
we find together with \eqref{S4} and the relations \eqref{E:a1}-\eqref{E:a3} of Lemma \ref{L:A1} (with $r=3/2,$ $n_0=3$, and $N_0=m-1$) that
\begin{equation}\label{G:2}
\begin{aligned}
 &\max\left\{\left\|\p_q^n\left(1+h_q^2\right)^{3/2}\right\|_\alpha,\left\|\p_q^n\left(1+h_q^2\right)^{5/2}\right\|_\alpha\right\}\leq K_7L^{n-3/2}(n-3)!,\\
&\max\left\{\left\|\p_q^n\left(h_q\left(1+h_q^2\right)^{1/2}\right)\right\|_\alpha, \left\|\p_q^n\left(h_q\left(1+h_q^2\right)^{3/2}\right)\right\|_\alpha\right\}\leq K_7L^{n-3/2}(n-3)!
\end{aligned}
\end{equation}
for all $3\leq n\leq m-1.$
On the other hand, the relation    \eqref{S2}  and the estimates \eqref{E:a1}-\eqref{E:a2} of Lemma \ref{L:A1} (with $r=3/2,$ $n_0=3$, and $N_0=m-1$) yield
\begin{equation}\label{G:3}
\begin{aligned}
 &\max\left\{\left\|\p_q^n\left(1/h_p^2\right)\right\|_\alpha, \left\|\p_q^n\left(1/h_p^3\right)\right\|_\alpha\right\}\leq K_8L^{n-3/2}(n-3)!
\end{aligned}
\end{equation}
for all $3 \leq n\leq m-1.$
In view of \eqref{S1'}-\eqref{G:3}  and of
\begin{align}\label{S1''}
 \|\p_q^nh\|_\alpha\leq K_9L^{n-3/2}(n-3)! \qquad\text{for $3\leq n\leq m,$}
\end{align}
the $C^{1+\alpha}-$norm of the functions $\varphi_m^1$ and $\varphi_m^2$ can be estimated by the same quantity.
More precisely, we have 
\begin{align*}
 \|\varphi_m^1\|_{1+\alpha}\leq C_6\left\|\sum_{k=1}^{m-1} \begin{pmatrix}m\\k\end{pmatrix}(\p_q^k(1+h_q^2)^{5/2})\p_q^{m-k}\frac{1}{h_p^2}\right\|_\alpha,
\end{align*}
and we split the sum in the latter sum as follows
\[
\sum_{k=1}^{m-1} \begin{pmatrix}m\\k\end{pmatrix}(\p_q^k(1+h_q^2)^{5/2})\p_q^{m-k}\frac{1}{h_p^2}=\left(\sum_{k=1}^2+\sum_{k=3}^{m-3}+\sum_{k=m-2}^{m-1}\right)\begin{pmatrix}m\\k\end{pmatrix}(\p_q^k(1+h_q^2)^{5/2})\p_q^{m-k}\frac{1}{h_p^2}.
\]
Recalling \eqref{G:2} and \eqref{G:3}, we get  
\begin{align*}
 \left\|\left(\sum_{k=1}^2 +\sum_{k=m-2}^{m-1}\right)\begin{pmatrix}m \\k\end{pmatrix}(\p_q^k(1+h_q^2)^{5/2})\p_q^{m-k}\frac{1}{h_p^2}\right\|_\alpha\leq K_{10}L^{m-5/2}(m-3)!.
\end{align*}
When estimating the middle sum we take advantage of \eqref{arg}, \eqref{G:2}, and \eqref{G:3} to find
\begin{align*}
 \left\| \sum_{k= 3}^{m-3}\begin{pmatrix}m\\k\end{pmatrix}(\p_q^k(1+h_q^2)^{5/2})\p_q^{m-k}\frac{1}{h_p^2}\right\|_\alpha
 \leq& \sum_{k= 3}^{m-3}\begin{pmatrix}m\\k\end{pmatrix}\left\|\p_q^k(1+h_q^2)^{5/2}\right\|_\alpha\left\|\p_q^{m-k}\frac{1}{h_p^2}\right\|_\alpha\\
 \leq&K_{11} \sum_{k= 3}^{m-3}\begin{pmatrix}m \\k\end{pmatrix}L^{k-3/2}(k-3)! L^{m-k-3/2}(m-k-3)!\\
 \leq &K_{11}L^{m-3}(m-3)!\sum_{k= 3}^{m-3}\frac{m^3}{(k-2)^3(m-k-2)^3}\\
 \leq &K_{11}L^{m-3}(m-3)!.
\end{align*}
The arguments being also true when estimating $\varphi_m^2,$ we conclude that
\begin{align}\label{F3}
 \sum_{i=1}^2\|\varphi^i_m\|_\alpha \leq K_{12}L^{m-5/2}(m-3)!.
\end{align}

Finally, recalling \eqref{S1'}-\eqref{G:3}, one can easily see that the norms  $\|\varphi_m^i\|_{1+\alpha}$,  $i\in\{3,4,5\},$ may  be bounded by the same quantity.
Indeed, we have that
\begin{align*}
 \|\varphi_m^3\|_{1+\alpha}\leq&K_{13}\left\| \sum_{k=0}^{m-2} \begin{pmatrix}m-1\\k\end{pmatrix}(\p_q^{k+1}h_p)\p_q^{m-k-1}\frac{1}{h_p^3}\right\|_\alpha,
\end{align*}
and we split the sum on the right-hand side of the latter inequality as follows
\begin{align*}
  \sum_{k=0}^{m-2} \begin{pmatrix}m-1\\k\end{pmatrix}(\p_q^{k+1}h_p)\p_q^{m-k-1}\frac{1}{h_p^3}=\left(\sum_{k=0}^1+\sum_{k=2}^{m-4}+\sum_{k=m-3}^{m-2}\right)\begin{pmatrix}m-1\\k\end{pmatrix}(\p_q^{k+1}h_p)\p_q^{m-k-1}\frac{1}{h_p^3}.
\end{align*}
The relations \eqref{S1'} and \eqref{G:3} imply that
\begin{align*}
 \left\|\left(\sum_{k=0}^1+ \sum_{k=m-3}^{m-2}\right)\begin{pmatrix}m-1\\k\end{pmatrix}(\p_q^{k+1}h_p)\p_q^{m-k-1}\frac{1}{h_p^3}\right\|_\alpha\leq K_{14}L^{m-5/2}(m-3)!
\end{align*}
and 
\begin{align*}
 \left\| \sum_{k= 2}^{m-4}\begin{pmatrix}m-1\\k\end{pmatrix}(\p_q^{k+1}h_p)\p_q^{m-k-1}\frac{1}{h_p^3}\right\|_\alpha
 \leq& \sum_{k= 2}^{m-4}\begin{pmatrix}m-1\\k\end{pmatrix}\|\p_q^{k+1}h_p\|_\alpha\left\|\p_q^{m-k-1}\frac{1}{h_p^3}\right\|_\alpha\\
 \leq&K_{15} L^{m-3}\sum_{k= 2}^{m-4}\begin{pmatrix}m-1 \\k\end{pmatrix} (k-2)! (m-k-4)!\\
 \leq &K_{15}L^{m-3}(m-3)!\sum_{k= 2}^{m-4}\frac{m^2}{(k-1)^2(m-k-3)^3}\\
 \leq &K_{15}L^{m-3}(m-3)!,
\end{align*}
meaning that
\begin{align}\label{F4}
 \sum_{i=3}^5\|\varphi^i_m\|_\alpha \leq K_{16}L^{m-5/2}(m-3)!.
\end{align}

Gathering \eqref{IND}, \eqref{F1}, \eqref{F2}, \eqref{F3}, and \eqref{F4}, we conclude that 
\begin{align}\label{E:D}
\|\p_q^mh\|_{1+\alpha}\leq K_2(1+K_6+K_{12}+K_{16})L^{m-5/2}(m-3)!,
\end{align}
  the constants $K_i$ being independent of $m$ and $L$. Therefore, we may  require,   additionally to \eqref{E:C}, that the  
constant $L$ should also satisfy
\begin{align}\label{E:constr}
L\geq K_2^2(1+K_6+K_{12}+K_{16})^2.
\end{align}
This additional restriction and \eqref{E:D} lead to the desired conclusion.
 \end{proof}

 We conclude the paper with the proof of our main existence result.
\begin{proof}[Proof of Theorem \ref{MT}]
The proof of Theorem \ref{MT} follows by combining the assertions of the Lemmas \ref{L:1} and \ref{L:Re} and that of the Theorems \ref{T:LB} and   \ref{MT2}. 
\end{proof}
 
 \vspace{0.5cm}
\hspace{-0.5cm}{\large \sc Acknowledgement} A.-V. Matioc was supported by the   ERC Advanced  Grant  ``Nonlinear studies of water flows with vorticity'' (NWFV).


\begin{thebibliography}{10}

\bibitem{A83}
H.~Amann.
\newblock {\em {Gew{\"o}hnliche {D}ifferentialgleichungen}}.
\newblock {de Gruyter Lehrbuch. [de Gruyter Textbook]}. Walter de Gruyter \&
  Co., Berlin, 1983.

\bibitem{LW12x}
H.~Chen, W.-X. Li, and L.-J. Wang.
\newblock {Regularity of traveling free surface water waves with vorticity}.
\newblock {\em J. Nonlinear Sci.}, 2013.
\newblock DOI 10.1007/s00332-013-9181-6.

\bibitem{C12}
D.~Clamond.
\newblock {Note on the velocity and related fields of steady irrotational
  two-dimensional surface gravity waves}.
\newblock {\em Philos. Trans. R. Soc. Lond. Ser. A Math. Phys. Eng. Sci.},
  370(1964):1572--1586, 2012.

\bibitem{Co06}
A.~Constantin.
\newblock {The trajectories of particles in {S}tokes waves}.
\newblock {\em Invent. Math.}, 166(3):523--535, 2006.

\bibitem{Con11}
A.~Constantin.
\newblock {\em {Nonlinear Water Waves with Applications to Wave-Current
  Interactions and Tsunamis}}, volume~81 of {\em {CBMS-NSF Conference Series in
  Applied Mathematics}}.
\newblock SIAM, Philadelphia, 2011.

\bibitem{CoEhWa07}
A.~Constantin, M.~Ehrnstr{\"o}m, and E.~Wahl{\'e}n.
\newblock {Symmetry of steady periodic gravity water waves with vorticity}.
\newblock {\em Duke Math. J.}, 140(3):591--603, 2007.

\bibitem{CoEs04_b}
A.~Constantin and J.~Escher.
\newblock {Symmetry of steady periodic surface water waves with vorticity}.
\newblock {\em J. Fluid Mech.}, 498(1):171--181, 2004.

\bibitem{AC11}
A.~Constantin and J.~Escher.
\newblock {Analyticity of periodic traveling free surface water waves with
  vorticity}.
\newblock {\em Ann. of Math.}, 173:559--568, 2011.

\bibitem{CoSt04}
A.~Constantin and W.~Strauss.
\newblock {Exact steady periodic water waves with vorticity}.
\newblock {\em Comm. Pure Appl. Math.}, 57(4):481--527, 2004.

\bibitem{CoSt10}
A.~Constantin and W.~Strauss.
\newblock {Pressure beneath a {S}tokes wave}.
\newblock {\em Comm. Pure Appl. Math.}, 63(4):533--557, 2010.

\bibitem{CS11}
A.~Constantin and W.~Strauss.
\newblock {Periodic traveling gravity water waves with discontinuous
  vorticity}.
\newblock {\em Arch. Ration. Mech. Anal.}, 202(1):133--175, 2011.

\bibitem{CV11}
A.~Constantin and E.~Varvaruca.
\newblock {Steady periodic water waves with constant vorticity: regularity and
  local bifurcation}.
\newblock {\em Arch. Ration. Mech. Anal.}, 199(1):33--67, 2011.

\bibitem{CR71}
M.~G. Crandall and P.~H. Rabinowitz.
\newblock {Bifurcation from simple eigenvalues}.
\newblock {\em J. Functional Analysis}, 8:321--340, 1971.

\bibitem{Esch-reg12}
J.~Escher.
\newblock {Regularity of rotational travelling water waves}.
\newblock {\em Philos. Trans. R. Soc. Lond. A}, 370:1602--1615, 2012.

\bibitem{EM13x}
J.~Escher and B.-V. Matioc.
\newblock {On the analyticity of periodic gravity water waves with integrable
  vorticity function}.
\newblock {\em Differential Integral Equations}, 2013.
\newblock to appear.

\bibitem{EG92}
L.~C. Evans and R.~F. Gariepy.
\newblock {\em {Measure theory and fine properties of functions}}.
\newblock {Studies in Advanced Mathematics}. CRC Press, Boca Raton, FL, 1992.

\bibitem{GT01}
D.~Gilbarg and N.~S. Trudinger.
\newblock {\em {Elliptic Partial Differential Equations of Second Order}}.
\newblock Springer Verlag, 2001.

\bibitem{DH07}
D.~Henry.
\newblock {Particle trajectories in linear periodic capillary and
  capillary-gravity deep-water waves}.
\newblock {\em J. Nonlinear Math. Phys}, 14:1--7, 2007.

\bibitem{Hen10}
D.~Henry.
\newblock {Analyticity of the streamlines for periodic travelling free surface
  capillary-gravity water waves with vorticity}.
\newblock {\em SIAM J. Math. Anal}, 42(6):3103--3111, 2010.

\bibitem{DH11a}
D.~Henry.
\newblock {Analyticity of the free surface for periodic travelling
  capillary-gravity water waves with vorticity}.
\newblock {\em J. Math. Fluid Mech.}, 14(2):249--254, 2012.

\bibitem{DH12}
D.~Henry.
\newblock {Regularity for steady periodic capillary water waves with
  vorticity}.
\newblock {\em Philos. Trans. R. Soc. Lond. A}, 370:1616--1628, 2012.

\bibitem{HH12}
H.-C. Hsu, Y.-Y. Chen, J.~R.~C. Hsu, and W.-J. Tseng.
\newblock {Nonlinear water waves on uniform current in {L}agrangian
  coordinates}.
\newblock {\em J. Nonlinear Math. Phys.}, 16(1):47--61, 2009.

\bibitem{MJ89}
M.~Jones.
\newblock {Small amplitude capillary-gravity waves in a channel of finite
  depth}.
\newblock {\em Glasgow Math. J.}, 31(2):141--160, 1989.

\bibitem{JT85}
M.~Jones and J.~Toland.
\newblock {The bifurcation and secondary bifurcation of capillary-gravity
  waves}.
\newblock {\em Proc. Roy. Soc. London Ser. A}, 399(1817):391--417, 1985.

\bibitem{JT86}
M.~Jones and J.~Toland.
\newblock {Symmetry and the bifurcation of capillary-gravity waves}.
\newblock {\em Arch. Rational Mech. Anal.}, 96(1):29--53, 1986.

\bibitem{Jon}
I.~G. Jonsson.
\newblock {\em {Wave-current interactions}}, volume~9.
\newblock Wiley, New York, 1990.

\bibitem{KO1}
J.~Ko and W.~Strauss.
\newblock {Large-amplitude steady rotational water waves}.
\newblock {\em Eur. J Mech. B Fluids}, 27:96--109, 2007.

\bibitem{KO2}
J.~Ko and W.~Strauss.
\newblock {Effect of vorticity on steady water waves}.
\newblock {\em J. Fluid Mech.}, 608:197--215, 2008.

\bibitem{L95}
A.~Lunardi.
\newblock {\em {Analytic semigroups and optimal regularity in parabolic
  problems}}.
\newblock {Progress in Nonlinear Differential Equations and their Applications,
  16}. Birkh{\"a}user Verlag, Basel, 1995.

\bibitem{CM12x}
C.~I. Martin.
\newblock {Local bifurcation and regularity for steady periodic
  capillary--gravity water waves with constant vorticity}.
\newblock {\em Nonlinear Anal. Real World Appl.}, (14):131--149, 2013.

\bibitem{CM13xx}
C.~I. Martin and B.-V. Matioc.
\newblock {Existence of capillary-gravity water waves with piecewise constant
  vorticity}.
\newblock arXiv:1302.5523.

\bibitem{CM13xxx}
C.~I. Martin and B.-V. Matioc.
\newblock {Existence of Wilton ripples for water waves with constant vorticity
  and capillary effects}.
\newblock {\em SIAM J. Appl. Math.}, 2013.

\bibitem{AM12x}
A.-V. Matioc.
\newblock {Steady internal water waves with a critical layer bounded by the
  wave surface}.
\newblock {\em J. Nonlinear Math. Phys.}, (19(1)):1250008, 21p, 2012.

\bibitem{MM13}
A.-V. Matioc and B.-V. Matioc.
\newblock {On the symmetry of periodic gravity water waves with vorticity}.
\newblock {\em Differential Integral Equations}, 26(1-2):129--140, 2013.

\bibitem{BM11}
B.-V. Matioc.
\newblock {Analyticity of the streamlines for periodic traveling water waves
  with bounded vorticity}.
\newblock {\em Int. Math. Res. Not.}, 17:3858--3871, 2011.

\bibitem{okamoto-shoji-01}
H.~Okamoto and M.~Sh{\=o}ji.
\newblock {\em {The mathematical theory of permanent progressive water-waves}}.
\newblock {Adv. Ser. Nonlinear Dynam. 20}. World Scientific Pub Co Inc, 2001.

\bibitem{O82}
K.~Okuda.
\newblock {Internal flow structure of short wind waves}.
\newblock {\em Journal of the Oceanographical Society of Japan}, 38:28--42,
  1982.

\bibitem{PC84}
C.~B. Pattiaratchi and M.~B. Collins.
\newblock {Sediment transport under waves and tidal currents: A case study from
  the northern Bristol Channel, U.K.}
\newblock {\em Marine Geology}, 56(1-4):27--40, 1984.

\bibitem{PB74}
O.~M. Phillips and M.~L. Banner.
\newblock {Wave breaking in the presence of wind drift and swell}.
\newblock {\em J. Fluid Mech.}, 66:625--640, 1974.

\bibitem{RS81}
J.~Reeder and M.~Shinbrot.
\newblock {On {W}ilton ripples. {II}. {R}igorous results}.
\newblock {\em Arch. Rational Mech. Anal.}, 77(4):321--347, 1981.

\bibitem{SP88}
A.~F. {Teles da Silva} and D.~H. Peregrine.
\newblock {Steep, steady surface waves on water of finite depth with constant
  vorticity}.
\newblock {\em J. Fluid Mech.}, 195:281--302, 1988.

\bibitem{Thom}
G.~Thomas and G.~Klopman.
\newblock {\em {Wave-current interactions in the nearshore region}}.
\newblock WIT, Southampton, United Kingdom, 1997.

\bibitem{Um12}
M.~Umeyama.
\newblock {Eulerian-{L}agrangian analysis for particle velocities and
  trajectories in a pure wave motion using particle image velocimetry}.
\newblock {\em Philos. Trans. R. Soc. Lond. Ser. A Math. Phys. Eng. Sci.},
  370(1964):1687--1702, 2012.

\bibitem{W06b}
E.~Wahl{\'e}n.
\newblock {Steady periodic capillary-gravity waves with vorticity}.
\newblock {\em SIAM J. Math. Anal.}, 38(3):921--943 (electronic), 2006.

\bibitem{WZ12}
G.~S. Weiss and G.~Zhang.
\newblock {A free boundary approach to two-dimensional steady capillary gravity
  water waves}.
\newblock {\em Arch. Ration. Mech. Anal.}, 203(3):747--768, 2012.

\end{thebibliography}
\end{document}